\documentclass[a4paper]{article}
\usepackage[margin=25mm]{geometry}
\usepackage{amsmath}
\usepackage{amsthm}
\usepackage{amssymb}
\usepackage{amsfonts}
\usepackage{latexsym}
\usepackage{mathtools}
\usepackage{array}
\usepackage{bm}
\usepackage{comment} 
\usepackage{authblk}

\usepackage{tikz}
\usetikzlibrary{arrows.meta, decorations.pathreplacing, positioning, calc, shapes.geometric, shapes, backgrounds}

\newtheorem{prop}{Proposition}[section]
\newtheorem{lem}[prop]{Lemma}
\newtheorem{thm}[prop]{Theorem}
\newtheorem{cor}[prop]{Corollary}

\theoremstyle{definition}
\newtheorem{dfn}[prop]{Definition}

\newtheorem{state}[prop]{Statement}

\numberwithin{prop}{section}
\numberwithin{equation}{section}

\newenvironment{lemproof}{
  
  \begin{proof}
}{
  \end{proof}
}

\title{Bisimulations in second-order arithmetic}
\author{Yuto Takeda, Keita Yokoyama}
\affil{Mathematical Institute, Tohoku University, Japan}

\usepackage[T1]{fontenc}

\begin{document}
\maketitle
%
%\titlerunning{Abbreviated paper title}
% If the paper title is too long for the running head, you can set
% an abbreviated paper title here
%

%
%\authorrunning{Y. Takeda et al.}
% First names are abbreviated in the running head.
% If there are more than two authors, 'et al.' is used.
%
%\institute{Tohoku University, Sendai, Miyagi, JPN\and
%\email{sho.shimomichi.s8@dc.tohoku.ac.jp}\and
%\email{yuto.takeda.t8@dc.tohoku.ac.jp}\and
%\email{keita.yokoyama.c2@dc.tohoku.ac.jp}}
%

%
\begin{abstract}
This paper investigates the logical strength of two theorems in modal propositional logic—the Hennessy–Milner theorem and the van Benthem characterization theorem—within the framework of second-order arithmetic. We demonstrate that the Hennessy–Milner theorem is equivalent to $\mathrm{ACA}_0$ over $\mathrm{RCA}_0$. For the van Benthem characterization theorem, we introduce three variants: the semantic, syntactic, and hybrid forms. We show that the semantic form is provable in $\mathrm{RCA}_0$, the syntactic form is provable in $\mathrm{PRA}$, and the hybrid form is equivalent to the weak completeness theorem for first-order logic over $\mathrm{RCA}_0$.
%In this paper, we investigate the strength of completeness theorems for modal propositional logic in second-order arithmetic. We show that the weak completeness theorem for modal propositional logic is provable in $\mathrm{RCA}_0$ and that $\mathrm{ACA}_0$ is equivalent over $\mathrm{RCA}_0$ to the strong completeness theorem for modal propositional logic. Moreover, we define a weak form of the strong completeness theorem, and we show that $\mathrm{WKL}_0$ is equivalent to this weak form.
%\keywords{second-order arithmetic \and reverse mathematics \and completeness theorems \and modal logic}
\end{abstract}
\section{Introduction}
%この論文では, 二階算術において, 様相論理における双模倣性の分析を行う. 特に, 様相論理において重要な定理である, Hennessy-Milner theoremとvan Benthem characterization theoremについて逆数学の視点から分析を行う. 双模倣とは, Kripkeモデルの世界の間の関係であり, 命題の真偽を保存し, 到達可能性関係を両方向から模倣するものである. 本論文では, $\mathrm{RCA}_0$上で, Hennessy-Milner theoremが体系$\mathrm{ACA}_0$と同値であることを示した。この結果は証明論における一結果にとどまらず、Hennessy-Milner theoremの十分条件であるimage-finitenessと可能世界間おける様相同値性のみでは, 計算可能な手法でクリプキモデル間の双模倣関係を構築することは（一般には）できないということを示しており, これは計算可能性理論の視点においても重要な結果であるといえる. またvan Benthem characterization theoremに関して、この主張が$\mathrm{RCA}_0$上$\mathrm{WKL}_0$で証明可能であることを示し、更なる詳細な分析を経て、原始再帰的関数を取り扱うだけの性能がある公理系$\mathrm{PRA}$においても証明が可能であることを示した. これはさらに弱い体系においてもこの定理が証明可能であることを示唆している.
In this paper, we investigate bisimulations for modal logic in second-order arithmetic. In particular, we analyze two fundamental theorems in modal logic—the Hennessy–Milner theorem~\cite{MR832336} and the van Benthem characterization theorem~\cite{van2014modal}—from the perspective of reverse mathematics. Several studies have analyzed logic within second-order arithmetic. For example, over $\mathrm{RCA}_0$, Simpson~\cite{MR2517689} showed that G\"odel’s completeness theorem is equivalent to $\mathrm{WKL}_0$ and Yamazaki~\cite{MR2010178} showed that the strong completeness theorem for intuitionistic predicate logic is equivalent to $\mathrm{ACA}_0$. %In \cite{MR2517689}, Simpson studied the relationships between several well-known theorems of classical logic in second-order arithmetic and showed that, over $\mathrm{RCA}_0$, G\"odel's completeness theorem is equivalent to $\mathrm{WKL}_0$. Similar studies have been conducted on intuitionistic logic, and in \cite{MR2010178}, Yamazaki showed that, over $\mathrm{RCA}_0$, the strong completeness theorem for intuitionistic predicate logic is equivalent to $\mathrm{ACA}_0$. 
For modal logic, we previously showed that the weak completeness theorem for propositional modal logic is provable in $\mathrm{RCA}_0$. We then showed that proving the strong completeness theorem via canonical models requires $\mathrm{ACA}_0$, whereas the strong completeness theorem itself is equivalent to $\mathrm{WKL}_0$ over $\mathrm{RCA}_0$~\cite{10.1007/978-3-031-95908-0_32}.%A bisimulation is a relation between worlds in Kripke models that preserves the truth of propositional variables and mirrors the accessibility relation in both directions.

We show that, over $\mathrm{RCA}_0$, the Hennessy–Milner theorem is equivalent to the system $\mathrm{ACA}_0$. This result is not merely a proof-theoretic observation: it also implies that, in general, the assumption of image-finiteness alone is insufficient to effectively construct a bisimulation between modally equivalent Kripke models. Hence, this equivalence is also significant from the viewpoint of computability theory. 

Regarding the van Benthem characterization theorem, we introduce and analyze three variants: the semantic, syntactic, and hybrid forms. First, we show that the semantic form is provable in $\mathrm{RCA}_0$. While van Benthem's original proof~\cite{van2014modal} relies on ultrafilter extensions—which are inherently difficult to handle in second-order arithmetic—we demonstrate that Otto's elementary approach \cite{Otto2004b, MR2258711} can be successfully refined within $\mathrm{RCA}_0$. Second, a closer analysis reveals that the syntactic form is provable even in $\mathrm{PRA}$. This not only suggests that the characterization theorem may hold in even weaker systems, but also implies that a primitive recursive process can effectively yield a characteristic formula from a given proof of bisimulation invariance, contributing a fresh insight to computability theory. Finally, we establish that the hybrid form is equivalent over $\mathrm{RCA}_0$ to the weak completeness theorem for first-order logic. By shifting to the hybrid setting, we illuminate a novel connection that yields an equivalence to a system strictly weaker than $\mathrm{WKL}_0$.

%Regarding the van Benthem characterization theorem, we introduce three variants: the semantic, syntactic, and hybrid forms. First, we prove that the semantic form is provable in $\mathrm{RCA}_0$. Van Benthem's original proof~\cite{van2014modal} used ultrafilter extensions, which are inherently difficult to handle within second-order arithmetic. In contrast, Otto's elementary proof \cite{Otto2004b, MR2258711} does not use ultrafilters, and the present result demonstrates that his approach can be refined within $\mathrm{RCA}_0$. Moreover, through a more detailed analysis, we show that the syntactic form is provable even in $\mathrm{PRA}$, a system strong enough to formalize reasoning about primitive recursive functions. This suggests that the van Benthem characterization theorem may be provable in even weaker systems. Furthermore, this result demonstrates that if there is a \emph{proof} of the bisimulation invariance of first-order logic formulas, a characteristic formula can be obtained via a primitive recursive process. This contributes a valuable insight to computability theory. Last, we show the hybrid form is equivalent to the weak completeness theorem for first-order logic over $\mathrm{RCA}_0$. While the relationship between the bisimulation invariance and the unsatisfiability of first-order formulas is well known, reconsidering this connection in a hybrid form has made it possible to establish an equivalence to the weak completeness theorem for first-order logic, which is strictly weaker than $\mathrm{WKL}_0$.

In Section 2, we show that the Hennessy-Milner theorem is equivalent to $\mathrm{ACA}_0$ over $\mathrm{RCA}_0$. In Section 3, we consider three variants of the van Benthem characterization theorem. First, we show that the semantic form is provable in $\mathrm{RCA}_0$. Specifically, we formalize the elementary proof of the theorem by Otto~\cite{Otto2004b, MR2258711} within $\mathrm{RCA}_0$. Moreover, we demonstrate that the syntactic form is provable in $\mathrm{PRA}$, and show that the hybrid form is equivalent to the weak completeness theorem for first-order logic. For the formalization of modal logic over $\mathrm{RCA}_0$, we refer to~\cite{10.1007/978-3-031-95908-0_32}.

%\subsection{Framework}
It uses the framework of subsystems of second-order arithmetic, with base theory $\mathrm{RCA}_0$. $\mathrm{RCA}_0$ consists of the axioms of Robinson arithmetic, together with the $\Sigma^0_1$-induction scheme and the $\Delta^0_1$-comprehension scheme. $\mathrm{WKL}_0$ consists of the axioms of $\mathrm{RCA}_0$ together with a formalization of weak K\"onig's lemma. $\mathrm{ACA}_0$ consists of the axioms of $\mathrm{RCA}_0$ together with the comprehension scheme for all arithmetic formulas. $\mathrm{PRA}$ is formulated in a language containing a symbol for each primitive recursive function. Its axioms consist of the usual axioms for equality, the induction scheme for quantifier-free formulas, and the defining equations of primitive recursive functions. It is known that $\mathrm{WKL}_0$ is a $\Pi^0_2$-conservative extension of $\mathrm{PRA}$ (see~\cite{MR2517689}). 

%\subsection{Bisimulations for modal logic}

\renewcommand{\thefootnote}{\fnsymbol{footnote}}
\section{The Hennessy-Milner theorem in second-order arithmetic}
In this section, we prove that the Hennessy-Milner theorem is equivalent to $\mathrm{ACA}_0$ over $\mathrm{RCA}_0$.  As usual, we identify finite sequences of symbols of the language for modal logic with natural numbers. 

Let $\mathrm{Prop}=\{p,q,r,\dots\}$ be an infinite set of propositional variables. In $\mathrm{RCA}_0$, the set of all atomic formulas, $\mathrm{Atm}=\mathrm{Prop}\cup\{\bot\}$, and the set of all modal formulas, $\mathrm{Fml}$, exist as the sets of natural numbers. 

We next formalize the Kripke semantics.
\begin{dfn}[$\mathrm{RCA}_0$~\cite{10.1007/978-3-031-95908-0_32}]
A \emph{Kripke model} is a tuple $M=(W, R, V)$ satisfying the following conditions:
\begin{itemize}
    \item $W\subseteq\mathbb{N}$ is a non-empty set,
    \item $R$ is a binary relation on $W$, i.e., $R\subseteq W\times W$,
    \item %$V$ is a function which assigns a truth value to each pair of a propositional formula and an element of $W$, i.e., 
$V : W\times\mathrm{Fml} \rightarrow\{0, 1\}$,
    \item $V(w, \bot)=0$ for any $w\in W$, 
    \item $V(w, \varphi\rightarrow\psi)=1-V(w, \varphi)(1-V(w, \psi))$ for any $w\in W$ and any $\varphi, \psi\in\mathrm{Fml}$, 
    \item $V(w, \Box\varphi)=1\iff\forall v\in W(wRv\rightarrow V(v, \varphi)=1)$ for any $w\in W$ and any $\varphi\in\mathrm{Fml}$.
\end{itemize}
A pair $(W,R)$ is called a \emph{frame}, and a function $V$ is called a \emph{valuation} on $(W,R)$. A \emph{pointed model} is a pair $(M, w)$ where $w\in W$. 
\end{dfn}

We also use the following notations.
Let $F$ be a frame, $M$ be a Kripke model, $w\in W$, and $\varphi\in\mathrm{Fml}$, we define
\begin{align*}
&M, w\Vdash\varphi\equiv V(w, \varphi)=1, \quad \quad M\Vdash\varphi\equiv(\forall w\in W)(V(w, \varphi)=1),\\
&F\Vdash\varphi\equiv(\forall V\text{: valuation on }F)(\forall w\in W)(V(w, \varphi)=1).
\end{align*}

In the above definition, the valuation $V$ should cover the truth values of all formulas. Within $\mathrm{RCA}_{0}$, one cannot extend a valuation on atomic formulas to a full valuation.

\begin{prop}[$\mathrm{RCA}_0$~\cite{MR2517689}]\label{acaeqijf}
$\mathrm{ACA}_0$ is equivalent to the statement that for each injective function $f:\mathbb{N}\rightarrow\mathbb{N}$, there exists a set $A$ such that $\forall x(x\in A\leftrightarrow\exists y(f(y)=x))$.
\end{prop}

\begin{state}\label{pextend}
For each frame $F=(W, R)$ and $v:W\times\mathrm{Prop}\rightarrow2$, there exists a valuation $V:W\times\mathrm{Fml}\rightarrow2$ such that $M=(W, R, V)$ is a Kripke model and $V$ is an extension of $v$.
\end{state}

\begin{prop}[$\mathrm{RCA}_0$]
Statement~\ref{pextend} implies $\mathrm{ACA}_0$.
\end{prop}

\begin{proof}
By Proposition~\ref{acaeqijf}, it suffices to show that for an injective function $f$, the range of $f$ exists. We take an injective function $f:\mathbb{N}\rightarrow\mathbb{N}$. Let $\mathrm{Prop}=\{p_{i}\}_{i\in\mathbb{N}}$, $W=\mathbb{N}$ and $R=W\times W$. Then we define $v:W\times\mathrm{Prop}\rightarrow2$ by the following: for all $p_{i}\in\mathrm{Prop}$ and $w\in W$, $v(w, p_{i})=1\Leftrightarrow f(w)=i$. By the assumption, there exists a full valuation $V: W\times\mathrm{Fml}\rightarrow2$ such that $M=(W, R, V)$ is a model and $V$ is an extension of $v$. Then we define $I=\{w\in W\mid V(w, \Diamond p_{w})=1\}$. $I$ is the range of $f$. In fact, for all $n$, $\exists m(f(m)=n)\Leftrightarrow\exists m(v(m, p_{n})=1)\Leftrightarrow V(n, \Diamond p_{n})=1\Leftrightarrow n\in I$.
\end{proof}

\begin{comment}
\begin{dfn}[$\mathrm{RCA}_0$]
Let $F=(W, R)$ be a frame.
\begin{itemize}
    \item $F$ is \emph{appropriate to $T$} if $R$ is reflexive,
    \item $F$ is \emph{appropriate to $B$} if $R$ is symmetric,
    \item $F$ is \emph{appropriate to $4$} if $R$ is transitive,
    \item $F$ is \emph{appropriate to $5$} if $R$ is Euclidean,
    \item $F$ is \emph{appropriate to $D$} if $R$ is serial,
    \item $F$ is \emph{appropriate to $.2$} if $R$ is directed,
    \item $F$ is \emph{appropriate to $L$} if $R$ is transitive and $W$ has no infinite ascending sequences by $R$.
\end{itemize}

In general, $F$ is \emph{appropriate to $\mathrm{\Sigma}\subseteq\{T, B, 4, 5, D, .2, L\}$} if for all $\varphi\in\mathrm{\Sigma}$, $F$ is appropriate to $\varphi$.
\end{dfn}
\end{comment}

Given a frame $F$, a set of formulas $\mathrm{\Gamma}$, and $\varphi\in\mathrm{Fml}$, we let
\begin{align*}
F\Vdash\mathrm{\Gamma}&\equiv(\forall\psi\in\mathrm{\Gamma})(F\Vdash\psi),\\
\mathrm{\Gamma}\Vdash\varphi&\equiv(\forall F\text{: frame and }\forall V\text{: valuation on }F)((F, V)\Vdash\mathrm{\Gamma}\rightarrow (F, V)\Vdash\varphi).
\end{align*}

\begin{dfn}[$\mathrm{RCA}_0$]
Let $M=(W, R, V)$ be a Kripke model. 
\begin{itemize}
    \item $M$ is image-finite if for all $w\in W$, $wR=\{v\in W\mid wRv\}$ is a finite set.
    \item $M$ is modally saturated if for every $w\in W$ and set of modal formulas $\Gamma$, the following condition holds: if $M, w\Vdash\Diamond\wedge\Gamma_{0}$ for all finite $\Gamma_{0}\subseteq\Gamma$, then there exists $u\in wR$ such that $M, u\Vdash\Gamma$.
\end{itemize}
\end{dfn}

\begin{dfn}[$\mathrm{RCA}_0$]
Let $M=(W, R, V)$ and $M^{\prime}=(W^{\prime}, R^{\prime}, V^{\prime})$ be two models. Two pointed models $(M, w)$ and $(M^{\prime}, w^{\prime})$ are modally equivalent (we denote $(M, w)\leftrightsquigarrow (M^{\prime}, w^{\prime})$) if for all $\varphi\in\mathrm{Fml}$, $V(w, \varphi)=1$ if and only if $V^{\prime}(w^{\prime}, \varphi)=1$.
\end{dfn}

\begin{dfn}[$\mathrm{RCA}_0$]
Let $M=(W, R, V)$ and $M^{\prime}=(W^{\prime}, R^{\prime}, V^{\prime})$ be two models. A non-empty binary relation $Z\subseteq W\times W^{\prime}$ is called a bisimulation between $M$ and $M^{\prime}$ (we denote $Z : M\,\underline{\xleftrightarrow{}}\,M^{\prime}$) if the following conditions are satisfied:
\begin{itemize}
    \item $\forall w\forall w^{\prime}(wZw^{\prime}\rightarrow\forall p\in\mathrm{Prop} (V(w, p)=1\leftrightarrow V^{\prime}(w^{\prime}, p)=1))$.
    \item $\forall w\forall w^{\prime}\forall v(wZw^{\prime}\land wRv\rightarrow\exists v^{\prime}(vZv^{\prime}\land w^{\prime}R^{\prime}v^{\prime}))$ (the forth condition).
    \item $\forall w\forall w^{\prime}\forall v^{\prime}(wZw^{\prime}\land w^{\prime}R^{\prime}v^{\prime}\rightarrow\exists v(vZv^{\prime}\land wRv))$ (the back condition).
\end{itemize}
Two pointed models $(M, w)$ and $(M^{\prime}, w^{\prime})$ are bisimilar (we denote $(M, w)\,\underline{\xleftrightarrow{}}\,(M^{\prime}, w^{\prime})$) if there exists a bisimulation $Z : M\,\underline{\xleftrightarrow{}}\,M^{\prime}$ such that $wZw^{\prime}$.
\end{dfn}

\begin{prop}[$\mathrm{RCA}_0$]\label{bisimimmeq}
For any two pointed models $(M, w)$ and $(M^{\prime}, w^{\prime})$, $(M, w)\,\underline{\xleftrightarrow{}}\,(M^{\prime}, w^{\prime})$ implies $(M, w)\leftrightsquigarrow (M^{\prime}, w^{\prime})$.
\end{prop}

\begin{proof}
Let $M=(W, R, V)$ and $M^{\prime}=(W^{\prime}, R^{\prime}, V^{\prime})$ be two models. Fix $w\in W$ and $w^{\prime}\in W^{\prime}$ and assume that $(M, w)\,\underline{\xleftrightarrow{}}\,(M^{\prime}, w^{\prime})$ holds, that is, there exists a bisimulation $Z : M\,\underline{\xleftrightarrow{}}\,M^{\prime}$ such that $wZw^{\prime}$. Then we show the following statement by induction on the length of the formula $\varphi$:
\begin{align*}
\forall w_0\forall w_1(w_0Zw_1\rightarrow(V(w_0, \varphi)=1\leftrightarrow V(w_1, \varphi)=1)).
\end{align*}
This can be shown using $\mathrm{I}\Sigma^0_1$.
\vspace{2mm}

\textit{Base case}: For all $p\in\mathrm{Prop}$, by the definition of the bisimulation, $V(w, p)=1\leftrightarrow V^{\prime}(w^{\prime}, p)=1$ holds.
\vspace{-2mm}

\textit{Induction step}: Take $\varphi\equiv\Box\psi$ and assume that for all $w_0\in W$ and $w_1\in W^{\prime}$ with $w_0Zw_1$, $V(w_0, \psi)=1\leftrightarrow V^{\prime}(w_1, \psi)=1$. Assume that $V(w_0, \Box\psi)=1$ holds. Then for all $v\in w_{0}R$, $V(v,\psi)=1$. Take $v_{1}\in w_{1}R^{\prime}$. By the back condition, there exists $v_0\in w_{0}R$ such that $v_{0}Zv_{1}$ holds. Therefore $V(v_0, \psi)=1$. By the assumption, $V^{\prime}(v_{1}, \psi)=1$, so $V^{\prime}(w_1, \Box\psi)=1$.

Next, we assume that $V^{\prime}(w_1, \Box\psi)=1$. Then for all $v\in w_{1}R^{\prime}$, $V^{\prime}(v, \psi)=1$. Take $v_{0}\in w_{0}R$. By the forth condition, there exists $v_1\in w_{1}R^{\prime}$ such that $v_{0}Zv_{1}$. Therefore $V^{\prime}(v_{1}, \psi)=1$. By the assumption, $V(v_{0}, \psi)=1$, so $V(w_{0}, \Box\psi)=1$.
\vspace{2mm}

By the above, this completes the proof of the forward direction. Therefore $(M, w)\leftrightsquigarrow (M^{\prime}, w^{\prime})$ holds.
\end{proof}

\begin{prop}[The Hennessy-Milner theorem, $\mathrm{ACA}_0$~\cite{MR1837791}]\label{HMoverACA0}
For any two image-finite pointed models $(M, w)$ and $(M^{\prime}, w^{\prime})$, $(M, w)\leftrightsquigarrow (M^{\prime}, w^{\prime})$ if and only if $(M, w)\,\underline{\xleftrightarrow{}}\,(M^{\prime}, w^{\prime})$.
\end{prop}

\begin{proof}
By Proposition~\ref{bisimimmeq}, it is enough to show that for any two image-finite pointed models $(M, w)$ and $(M^{\prime}, w^{\prime})$, $(M, w)\leftrightsquigarrow (M^{\prime}, w^{\prime})$ implies $(M, w)\,\underline{\xleftrightarrow{}}\,(M^{\prime}, w^{\prime})$.

Let $M=(W, R, V)$ and $M^{\prime}=(W^{\prime}, R^{\prime}, V^{\prime})$ be two image-finite models. We assume that $(M, w)\leftrightsquigarrow (M^{\prime}, w^{\prime})$ holds. Then we define the following binary relation $Z\subseteq W\times W^{\prime}$ by $\Sigma^0_1$-comprehension:
\begin{align*}
Z=\{(w_0, w_1)\in W\times W^{\prime}\mid (M, w_{0})\leftrightsquigarrow (M^{\prime}, w_{1})\}.
\end{align*}
Note that $(M, w_{0})\leftrightsquigarrow (M^{\prime}, w_{1})$ is expressed by a $\Pi^0_1$-sentence. Then we show that $Z$ is a bisimulation between $M$ and $M^{\prime}$.
\vspace{2mm}

\textit{Base condition}: Take $(w_0, w_1)\in Z$ and $p\in\mathrm{Prop}$. By definition, $V(w_0, p)=1\leftrightarrow V^{\prime}(w_{1}, p)=1$ holds.
\vspace{2mm}

\textit{Forth condition}: We show the forth condition. Take $(w_0, w_1)\in Z$ and $v_{0}\in w_{0}R$. Towards a contradiction, assume that there is no $v_{1}$ such that $w_{1}R^{\prime}v_{1}$ and $v_{0}Zv_{1}$. By $V(w_{0}, \Diamond\top)=1$ and $w_{0}Zw_{1}$, $V^{\prime}(w_{1}, \Diamond\top)=1$ holds ($\top\equiv\neg\bot$). Therefore $w_{1}R^{\prime}$ is non-empty. Furthermore, as $M^{\prime}$ is image-finite, $w_{1}R^{\prime}$ must be finite, so we write $w_{1}R^{\prime}=\{v^{\prime}_1,\dots,v^{\prime}_n\}$. By the assumption, for every $v^{\prime}_i\in w_{1}R^{\prime}$, there exists $\psi_i\in\mathrm{Fml}$ such that $V(v_0, \psi_i)=1$ and $V^{\prime}(v^{\prime}_i, \psi_i)=0$. It follows that
\begin{align*}
V(w_0, \Diamond(\psi_1\land\dots\land\psi_n))=1\land V^{\prime}(w_1, \Diamond(\psi_1\land\dots\land\psi_n))=0.
\end{align*}
This contradicts our assumption that $w_{0}Zw_{1}$.
\vspace{2mm}

The back condition can be proved similarly to the forth condition. Therefore, $Z$ is a bisimulation. By definition, $wZw^{\prime}$, so $(M, w)\,\underline{\xleftrightarrow{}}\,(M^{\prime}, w^{\prime})$ holds.
\end{proof}

In fact, the Hennessy-Milner theorem for modally saturated models is also provable in $\mathrm{ACA}_0$.

\begin{prop}[The Hennessy-Milner theorem for modally saturated models, $\mathrm{ACA}_0$~\cite{MR1837791}]\label{HMMSoverACA0}
For any two modally saturated pointed models $(M, w)$ and $(M^{\prime}, w^{\prime})$, $(M, w)\leftrightsquigarrow (M^{\prime}, w^{\prime})$ if and only if $(M, w)\,\underline{\xleftrightarrow{}}\,(M^{\prime}, w^{\prime})$.
\end{prop}

\begin{proof}
By Proposition~\ref{bisimimmeq}, it is enough to show that for any two modally saturated pointed models $(M, w)$ and $(M^{\prime}, w^{\prime})$, $(M, w)\leftrightsquigarrow (M^{\prime}, w^{\prime})$ implies $(M, w)\,\underline{\xleftrightarrow{}}\,(M^{\prime}, w^{\prime})$.

Let $M=(W, R, V)$ and $M^{\prime}=(W^{\prime}, R^{\prime}, V^{\prime})$ be modally saturated models. Then we define the following binary relation $Z\subseteq W\times W^{\prime}$ by $\Sigma^0_1$-comprehension:
\begin{align*}
Z=\{(w_0, w_1)\in W\times W^{\prime}\mid (M, w_{0})\leftrightsquigarrow (M^{\prime}, w_{1})\}.
\end{align*}

Then we show that $Z$ is a bisimulation between $M$ and $M^{\prime}$.
\vspace{2mm}

\textit{Base condition}: the atomic condition is obvious. %take $(w_0, w_1)\in Z$ and $p\in\mathrm{Prop}$. By the definition, $V(w_0, p)=1\leftrightarrow V^{\prime}(w_{1}, p)=1$ holds.
\vspace{2mm}

\textit{Forth condition}: we show the forth condition. Take $(w_0, w_1)\in Z$ and $v_{0}\in w_{0}R$. Let $\Gamma=\mathrm{Th}_{\mathrm{ML}}(M, v_{0})=\{\varphi\in\mathrm{Fml}\mid V(v_{0}, \varphi)=1\}$. For finite $\Gamma_{0}\subseteq\Gamma$, $V(w_{0}, \Diamond\wedge\Gamma_{0})=1$ holds, so $V^{\prime}(w_{1}, \Diamond\wedge\Gamma_{0})=1$  also holds. By modal saturation, there exists $v\in w_{1}R^{\prime}$ such that $M^{\prime}, v\Vdash\Gamma$. Then for all $\varphi\in\mathrm{Fml}$, $V(v_{0},\varphi)=1\leftrightarrow V^{\prime}(v,\varphi)=1$. Thus $v_{0}Zv$ holds.
\vspace{2mm}

The back condition can be proved similarly to the forth condition. Therefore, $Z$ is a bisimulation. By definition, $wZw^{\prime}$, so $(M, w)\,\underline{\xleftrightarrow{}}\,(M^{\prime}, w^{\prime})$ holds.
\end{proof}

\begin{thm}\label{HMT}
The following statements are equivalent over $\mathrm{RCA}_0$:
\begin{enumerate}
    \item $\mathrm{ACA}_0$,
    \item the Hennessy-Milner theorem for modally saturated models,
    \item the Hennessy-Milner theorem.
\end{enumerate}
\end{thm}

\begin{proof}
$(1)$ implies $(2)$ by Theorem~\ref{HMMSoverACA0}, and $(2)$ implies $(3)$ since every image-finite model is modally saturated. Thus, it is enough to show that the Hennessy-Milner theorem implies $\mathrm{ACA}_0$ over $\mathrm{RCA}_0$.
\vspace{1mm}

By Proposition~\ref{acaeqijf}, we take an injective function $f:\mathbb{N}\rightarrow\mathbb{N}$. Then we show that the range of $f$ exists. First, we define two image-finite models $M=(W, R, V)$ and $M^{\prime}=(W^{\prime}, R^{\prime}, V^{\prime})$. Let $\mathrm{Prop}=\{p_{i, j}\}_{i, j\in\mathbb{N}}\cup\{q_{j}\}_{j\in\mathbb{N}}\cup\{r_{j}\}_{j\in\mathbb{N}}\cup\{p_{*}\}$. Let $A, B, C, D, X, \{*\}$ be sets of symbols defined as $A=\{a_{i, j}\}_{i, j\in\mathbb{N}}$, $B=\{b_{i, j}\}_{i, j\in\mathbb{N}}$, $C=\{c_{j}\}_{j\in\mathbb{N}}$, $D=\{d_{j}\}_{j\in\mathbb{N}}$, and $X=\{x_{i, j}\}_{i, j\in\mathbb{N}}$ where $A, B, C, D, X$ are pairwise distinct. Then we define the sets of worlds $W$ and $W^{\prime}$:
\begin{align*}
W=A\cup B\cup C\cup D\cup X\cup\{*\},\,W^{\prime}=W\setminus X.
\end{align*}
Furthermore, we define  the accessibility relations $R\subseteq W\times W$: for every $w, v\in W$,
\begin{equation*}
\begin{split}
wRv\Leftrightarrow&\,(w=d_{j}\in D\land v=d_{j+1}\in D)\\&\lor(w=*\land v=d_{0}\in D)\\&\lor(w=d_{j}\in D\land v=c_{j}\in C)\\&\lor(w=c_{j}\in C\land v=a_{0, j}\in A\land f(0)\not=j)\\
&\lor(w=c_{j}\in C\land v=b_{i, j}\in B\land f(i)=j)\\%\land \forall k<i(f(k)\not=j))\\
&\lor(w=c_{j}\in C\land v=x_{0, j}\in X)\\
&\lor(w=a_{i, j}\in A\land v=a_{i+1, j}\in A\land f(i+1)\not=j)\\
&\lor(w=b_{i, j}\in B\land v=b_{i+1, j}\in B\land \exists m\leq i(f(m)=j))\\
&\lor(w=x_{i, j}\in X\land v=x_{i+1, j}\in X).\\
\end{split}
\end{equation*}
And we define $R^{\prime}=R\restriction(W^{\prime}\times W^{\prime})$. Moreover, we define the valuation for propositions $v:W\times\mathrm{Prop}\rightarrow 2$: for every $w\in W$ and $p\in\mathrm{Prop}$,
\begin{equation*}
\begin{split}
v(w, p)=1\Leftrightarrow&\,(w=a_{i, j}\in A\land\forall m\leq i(f(m)\not=j)\land p=p_{i, j})\\
&\lor(w=x_{i, j}\in X\land p=p_{i, j})\\
&\lor(w=b_{k, j}\in B\land\exists m\leq k(f(m)=j\land m+i=k)\land p=p_{i, j})\\
&\lor(w=c_{j}\in C\land p=q_{j})\\
&\lor(w=d_{j}\in D\land p=r_{j})\\
&\lor(w=*\land p=p_{*}).
\end{split}
\end{equation*}

%For example, $M^{\prime}$ is the model as follows:
As an example, consider the model $M^{\prime}$ defined below:
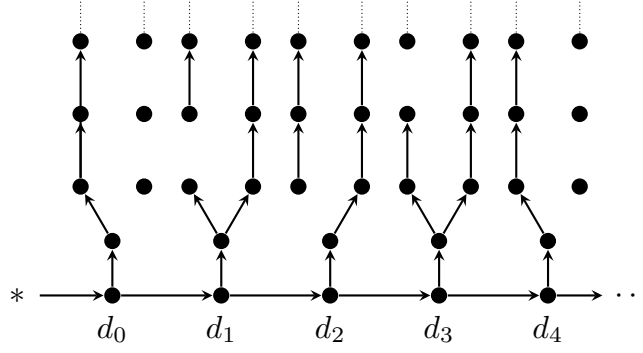
\begin{center}
\begin{minipage}{1.0\textwidth}
\centering
%%%%%%%%%%%%%%%%%%%%%%
%%  右　の　図      %%
%%  （今は左図を複製）
%%%%%%%%%%%%%%%%%%%%%%
\begin{tikzpicture}[>=stealth, scale=1.2, transform shape]
\tikzset{
  dot/.style={circle, fill, inner sep=1.8pt},
  vert/.style={->, thick},
  horiz/.style={->, thick}
}

% --- 基底ノードの x 座標 ---
\def\xA{0}
\def\xB{1.2}
\def\xC{2.4}
\def\xD{3.6}
\def\xE{4.8}

% --- 基底ノード ---
\node[dot] (b0) at (\xA,0) {};
\node[dot] (b1) at (\xB,0) {};
\node[dot] (b2) at (\xC,0) {};
\node[dot] (b3) at (\xD,0) {};
\node[dot] (b4) at (\xE,0) {};

% --- 基底ノードの下にノード追加 ---
\foreach \i/\x in {0/\xA,1/\xB,2/\xC,3/\xD,4/\xE}{
  \node[dot, label=below:$d_\i$] (bu\i) at (\x,-0.6) {};
  \draw[vert] (bu\i) -- (b\i);
}

% --- 基底ノードの下の水平基底線（区間ごとの矢印） ---
\draw[horiz] (bu0) -- (bu1);
\draw[horiz] (bu1) -- (bu2);
\draw[horiz] (bu2) -- (bu3);
\draw[horiz] (bu3) -- (bu4);
\draw[horiz] (bu4) -- ++(0.6,0) node[right] {$\cdots$};

\draw[<-, thick] (bu0) -- ++(-0.8,0);
\node at (-1.05,-0.6) {$*$};

% --- 水平基底線（区間ごとの矢印） ---
%\draw[horiz] (b0) -- (b1);
%\draw[horiz] (b1) -- (b2);
%\draw[horiz] (b2) -- (b3);
%\draw[horiz] (b3) -- (b4);
%\draw[horiz] (b4) -- ++(0.6,0) node[right] {$\cdots$};

% 左の incoming (p)
%\draw[<-, thick] (b0) -- ++(-0.8,0);
%\node at (-1.05,0) {$*$};

% --- 各基底から左右にフォーク ---
\foreach \i/\x in {0/\xA,1/\xB,2/\xC,3/\xD,4/\xE}{
  \pgfmathsetmacro{\xl}{\x-0.35}
  \pgfmathsetmacro{\xr}{\x+0.35}
  \node[dot] (l\i1) at (\xl,0.6) {};
  \node[dot] (r\i1) at (\xr,0.6) {};
  %\draw[vert] (b\i) -- (l\i1);
%  \draw[vert] (b\i) -- (r\i1);

  \foreach \j in {2,3}{
    \pgfmathsetmacro{\y}{0.6 + 0.8*(\j-1)}
    \node[dot] (l\i\j) at (\xl,\y) {};
    \node[dot] (r\i\j) at (\xr,\y) {};
%    \draw[vert] (l\i\the\numexpr\j-1\relax) -- (l\i\j);
%    \draw[vert] (r\i\the\numexpr\j-1\relax) -- (r\i\j);
  }
  \draw[densely dotted] (l\i3) -- ++(0,0.5);
  \draw[densely dotted] (r\i3) -- ++(0,0.5);
}

\draw[vert] (b0) -- (l01);
\draw[vert] (b1) -- (l11);
%\draw[vert] (b2) -- (l21);
\draw[vert] (b3) -- (l31);
\draw[vert] (b4) -- (l41);
%%%%%%%%%%%%%%%%%%%%%%%%%%%%%%%%%%%%%%%
%\draw[vert] (b0) -- (r01);
\draw[vert] (b1) -- (r11);
\draw[vert] (b2) -- (r21);
\draw[vert] (b3) -- (r31);

%\draw[vert] (r01) -- (r02);
%\draw[vert] (r02) -- (r03);

\draw[vert] (r11) -- (r12);
\draw[vert] (r12) -- (r13);

\draw[vert] (r21) -- (r22);
\draw[vert] (r22) -- (r23);

\draw[vert] (r31) -- (r32);
\draw[vert] (r32) -- (r33);
%%%%%%%%%%%%%%%%%%%%%%%%%%%%%%%%%%%%%%%%%%%%%%%%%%

\draw[vert] (l0\the\numexpr1\relax) -- (l02);
\draw[vert] (l0\the\numexpr1\relax) -- (l03);
%\draw[vert] (l1\the\numexpr1\relax) -- (l13);
\draw[vert] (l2\the\numexpr2\relax) -- (l23);
\draw[vert] (l3\the\numexpr1\relax) -- (l32);
%\draw[vert] (l3\the\numexpr2\relax) -- (l33);
\draw[vert] (l4\the\numexpr1\relax) -- (l42);
\draw[vert] (l4\the\numexpr2\relax) -- (l43);

\draw[vert] (l12) -- (l13);
\draw[vert] (l21) -- (l22);

\node at (0,0) {};

\end{tikzpicture}
\end{minipage}
\end{center}
Here, each arrow represents the accessibility relation, and each node represents a world. We assume that each node satisfies a unique propositional variable. The left branches of the forks belong to $A$, while the right branches belong to $B$. From this construction, it follows that $\forall x(f(x)\not=0)$, $\forall x(f(x)\not=4)$, $f(1)=1$, $f(0)=2$, and $f(2)=3$.
%Each arrow is the accessibility relation, and each node is the world. Each node satisfies unique propositional variables. The left sides of forks are in $A$, and the right sides of forks are in $B$. Then we find that $\not\exists x(f(x)=0)$, $\not\exists x(f(x)=4)$, $f(1)=1$, $f(0)=2$, and $f(2)=3$.

Also, $M$ is the following model:
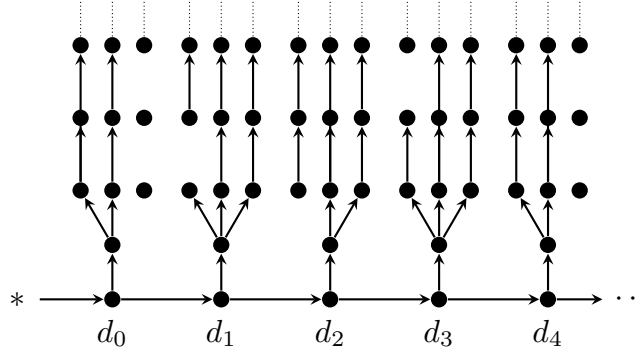
\begin{center}
\begin{minipage}{1.0\textwidth}
\centering
%%%%%%%%%%%%%%%%%%%%%%
%%  左　の　図      %%
%%%%%%%%%%%%%%%%%%%%%%
\begin{tikzpicture}[>=stealth, scale=1.2, transform shape]

\tikzset{
  dot/.style={circle, fill, inner sep=1.8pt},
  vert/.style={->, thick},
  horiz/.style={->, thick}
}

% --- 基底ノードの x 座標 ---
\def\xA{0}
\def\xB{1.2}
\def\xC{2.4}
\def\xD{3.6}
\def\xE{4.8}

% --- 基底ノード ---
\node[dot] (b0) at (\xA,0) {};
\node[dot] (b1) at (\xB,0) {};
\node[dot] (b2) at (\xC,0) {};
\node[dot] (b3) at (\xD,0) {};
\node[dot] (b4) at (\xE,0) {};

% --- 基底ノードの下にノード追加 ---
\foreach \i/\x in {0/\xA,1/\xB,2/\xC,3/\xD,4/\xE}{
  \node[dot, label=below:$d_\i$] (bu\i) at (\x,-0.6) {};
  \draw[vert] (bu\i) -- (b\i);
}

% --- 基底ノードの下の水平基底線（区間ごとの矢印） ---
\draw[horiz] (bu0) -- (bu1);
\draw[horiz] (bu1) -- (bu2);
\draw[horiz] (bu2) -- (bu3);
\draw[horiz] (bu3) -- (bu4);
\draw[horiz] (bu4) -- ++(0.6,0) node[right] {$\cdots$};

\draw[<-, thick] (bu0) -- ++(-0.8,0);
\node at (-1.05,-0.6) {$*$};

% --- 各基底からの構造 ---
\foreach \i/\x in {0/\xA,1/\xB,2/\xC,3/\xD,4/\xE}{
  \pgfmathsetmacro{\xl}{\x-0.35}
  \pgfmathsetmacro{\xr}{\x+0.35}
  \node[dot] (l\i1) at (\xl,0.6) {};
  \node[dot] (r\i1) at (\xr,0.6) {};
%  \draw[vert] (b\i) -- (r\i1);

  \node[dot] (c\i) at (\x,0.6) {};
  \draw[vert] (b\i) -- (c\i);

  \foreach \j in {2,3}{
    \pgfmathsetmacro{\y}{0.6 + 0.8*(\j-1)}
    \node[dot] (l\i\j) at (\xl,\y) {};
    \node[dot] (r\i\j) at (\xr,\y) {};
    \node[dot] (c\i\j) at (\x,\y) {};
%    \draw[vert] (r\i\the\numexpr\j-1\relax) -- (r\i\j);
  }
  \draw[densely dotted] (l\i3) -- ++(0,0.5);
  \draw[densely dotted] (r\i3) -- ++(0,0.5);
  \draw[densely dotted] (c\i3) -- ++(0,0.5);
}

\draw[vert] (b0) -- (l01);
\draw[vert] (b1) -- (l11);
%\draw[vert] (b2) -- (l21);
\draw[vert] (b3) -- (l31);
\draw[vert] (b4) -- (l41);

\draw[vert] (c0) -- (c02);
\draw[vert] (c02) -- (c03);

\draw[vert] (c1) -- (c12);
\draw[vert] (c12) -- (c13);

\draw[vert] (c2) -- (c22);
\draw[vert] (c2) -- (c23);

\draw[vert] (c3) -- (c32);
\draw[vert] (c3) -- (c33);

\draw[vert] (c4) -- (c42);
\draw[vert] (c4) -- (c43);
%%%%%%%%%%%%%%%%%%%%%%%%%%%%%%%%%%%%%%%
%\draw[vert] (b0) -- (r01);
\draw[vert] (b1) -- (r11);
\draw[vert] (b2) -- (r21);
\draw[vert] (b3) -- (r31);

%\draw[vert] (r01) -- (r02);
%\draw[vert] (r02) -- (r03);

\draw[vert] (r11) -- (r12);
\draw[vert] (r12) -- (r13);

\draw[vert] (r21) -- (r22);
\draw[vert] (r22) -- (r23);

\draw[vert] (r31) -- (r32);
\draw[vert] (r32) -- (r33);
%%%%%%%%%%%%%%%%%%%%%%%%%%%%%%%%%%%%%%%%%%%%%%%%%%

%%%%%%%%%%%%%%%%%%%%%%%%%%%%%%%%%%%%%%%%%%%%%%%%%
\draw[vert] (l0\the\numexpr1\relax) -- (l02);
\draw[vert] (l0\the\numexpr1\relax) -- (l03);
%\draw[vert] (l1\the\numexpr1\relax) -- (l13);
\draw[vert] (l2\the\numexpr2\relax) -- (l23);
\draw[vert] (l3\the\numexpr1\relax) -- (l32);
\draw[vert] (l4\the\numexpr1\relax) -- (l42);
\draw[vert] (l4\the\numexpr2\relax) -- (l43);

\draw[vert] (l12) -- (l13);
\draw[vert] (l21) -- (l22);

\end{tikzpicture}
\end{minipage}
\end{center}
Then every central path (consisting of nodes in $X$) is an infinite $R$-path.

In general, within $\mathrm{RCA}_{0}$, one cannot extend a valuation on atoms to a valuation on all formulas. But in this case, we can extend $v$ to the full valuation $V: W\times\mathrm{Fml}\rightarrow2$ step by step using only $\mathrm{I}\Sigma^0_1$ (that is, without the arithmetical comprehension axiom, $\mathrm{ACA}$). We demonstrate this below.

First, we define the following valuations
\begin{align*}
&F:(A\cup B\cup X)\times\mathrm{Fml}\rightarrow2,
&&G:(A\cup B\cup X\cup C)\times\mathrm{Fml}\rightarrow2\\
&H:(A\cup B\cup X\cup C\cup D)\times\mathrm{Fml}\rightarrow2,
&&V:W\times\mathrm{Fml}\rightarrow2
\end{align*}
such that $V$ is an extension of $H$, $H$ is an extension of $G$, and $G$ is an extension of $F$.
\vspace{2mm}

We write $R_{A\cup B\cup X}, R_{A\cup B\cup X\cup C}$, and $R_{A\cup B\cup X\cup C \cup D}$ for the restrictions of the relation $R$ to $A\cup B\cup X, A\cup B\cup X\cup C$, and $A\cup B\cup X\cup C \cup D$, respectively.
\vspace{2mm}

We define the function $F:(A\cup B\cup X)\times\mathrm{Fml}\rightarrow2$ by induction on the length of the formula. For all $w\in A\cup B\cup X$ and $p\in\mathrm{Prop}$, $F(w, p)\coloneq v(w, p)$. Next, fix $w\in A\cup B\cup X$ and $\varphi\in\mathrm{Fml}$. Assume that for all $w\in A\cup B\cup X$, $F(w, \varphi)$ is defined. Then
\begin{equation*}
\begin{split}
F(w, \Box\varphi)=1\Leftrightarrow&\,(w=a_{i, j}\in A\land f(i+1)\not=j\rightarrow F(a_{i+1, j}, \varphi)=1))\\
&\land(w=b_{i, j}\in B\land\exists m\leq i(f(m)=j)\rightarrow F(b_{i+1, j}, \varphi)=1)\\
&\land(w=x_{i, j}\in X\rightarrow F(x_{i+1, j}, \varphi)=1).
\end{split}
\end{equation*}

We can define such an $F$ using $\mathrm{I}\Sigma^0_1$. Next, we should check the following lemma.

\begin{lem}[$\mathrm{RCA}_0$]
The $F$ constructed above is a valuation for $(A\cup B\cup X,R_{A\cup B\cup X})$, that is, $F$ satisfies the following conditions:
\begin{itemize}
    \item $F(w, \bot)=0$ for any $w\in A\cup B\cup X$, 
    \item $F(w, \varphi\rightarrow\psi)=1-F(w, \varphi)(1-F(w, \psi))$ for any $w\in A\cup B\cup X$ and any $\varphi, \psi\in\mathrm{Fml}$, 
    \item $F(w, \Box\varphi)=1\iff\forall w^{\prime}\in A\cup B\cup X(wRw^{\prime}\rightarrow F(w^{\prime}, \varphi)=1)$ for any $w\in A\cup B\cup X$ and any $\varphi\in\mathrm{Fml}$.
\end{itemize}
\end{lem}

\begin{lemproof}
It is enough to check the last condition. Fix $w\in A\cup B\cup X$ and $\varphi\in\mathrm{Fml}$. It suffices to consider the case of the form $\Box\varphi$.
\vspace{2mm}

First, assume that $F(w, \Box\varphi)=1$. Take $w^{\prime}\in A\cup B\cup X$ with $wRw^{\prime}$. Consider the case of $w=a_{i, j}\in A$. If $a_{i, j}$ has a successor of $R$, its successor is only $a_{i+1, j}$ (by the definition of $R$). Therefore $w^{\prime}=a_{i+1, j}$. And if $f(i+1)\not=j$, then $a_{i, j}Ra_{i+1, j}$, otherwise $\neg(a_{i, j}Ra_{i+1, j})$. So, $f(i+1)\not=j$ holds. Therefore, $F(w^{\prime}, \varphi)=1$ holds. For the case of $w=b_{i, j}\in B$ and $w=x_{i, j}\in X$, we also show that $F(w^{\prime}, \varphi)=1$ holds. 
\vspace{2mm}

Next, assume that $F(w, \Box\varphi)=0$. If $w=a_{i, j}\in A\land f(i+1)\not=j$, then $a_{i, j}Ra_{i+1, j}$.  If $w=b_{i, j}\in B\land\exists m\leq i(f(m)=j)$, then $b_{i, j}Rb_{i+1, j}$. If $w=x_{i, j}\in X$, then $x_{i, j}Rx_{i+1, j}$. So, there exists $w^{\prime}\in A\cup B\cup X$ with $wRw^{\prime}$ such that $F(w^{\prime}, \varphi)=0$.
\end{lemproof}

Next, we define the function $G$ by induction on the length of the formula. For $\varphi\in\mathrm{Fml}$, $\mathrm{deg}(\varphi)$ is the number of modal operators in $\varphi$. For all $w\in A\cup B\cup X\cup C$ and $p\in\mathrm{Prop}$, $G(w, p)\coloneq v(w, p)$. Next, fix $w\in A\cup B\cup X\cup C$ and $\varphi\in\mathrm{Fml}$. Then
\begin{equation*}
\begin{split}
G(w, \Box\varphi)=1\Leftrightarrow&\,(w\notin C\rightarrow F(w, \Box\varphi)=1)\\
&\land(w=c_{j}\in C\land f(0)=j\rightarrow F(b_{0, j}, \varphi)=1\land F(x_{0, j}, \varphi)=1)\\
&\land\bigl(w=c_{j}\in C\land(0<\exists m\leq\mathrm{deg}(\varphi)f(m)=j)\\
&\rightarrow F(a_{0, j}, \varphi)=1\land F(b_{m, j}, \varphi)=1\land F(x_{0, j}, \varphi)=1\bigr)\\
&\land(w=c_{j}\in C\land(\forall m\leq\mathrm{deg}(\varphi)f(m)\not=j)\rightarrow F(a_{0, j}, \varphi)=1\land F(x_{0, j}, \varphi)=1).
\end{split}
\end{equation*}

Clearly, $G$ is an extension of $F$, and we can define such a $G$ using $\mathrm{I}\Sigma^0_1$. 

To show that $G$ is a valuation for $(A\cup B\cup X\cup C,R_{A\cup B\cup X\cup C})$, we show the following lemma.

\begin{lem}[$\mathrm{RCA}_0$]\label{A}
Assume that $\exists s>k(f(s)=j)$. Then for all $\varphi\in\mathrm{Fml}$ with $\mathrm{deg}(\varphi)\leq s-1$, $F(a_{0, j}, \varphi)=1\leftrightarrow F(b_{s, j}, \varphi)=1$ holds.
\end{lem}

\begin{lemproof}
Take $\varphi\in\mathrm{Fml}$ with $\mathrm{deg}(\varphi)\leq s-1$. We show the following statement by induction on $l\leq k$.
\begin{align*}
\forall l\leq s-1\forall\psi\in\mathrm{Fml}(\mathrm{deg}(\psi)\leq l\rightarrow(F(a_{s-1-l, j}, \psi)=1\leftrightarrow F(b_{2s-1-l, j}, \psi)=1).
\end{align*}

\textit{Base case}: For all $p\in\mathrm{Prop}$, $(F(a_{s-1, j}, p)=1\leftrightarrow F(b_{2s-1, j}, p)=1)$ holds by the definition of $v$. 
\vspace{2mm}

\textit{Induction step}: Fix $l<s-1$ and assume that for all $\psi$ with $\mathrm{deg}(\psi)\leq l$, $F(a_{s-1-l, j}, \psi)=1\leftrightarrow F(b_{2s-1-l, j}, \psi)=1$ holds. Take $\theta=\Box\psi$ with $\mathrm{deg}(\theta)\leq l+1$. Assume that $F(a_{s-1-(l+1), j}, \theta)=1$. Take $b\in b_{2s-1-(l+1), j}R$. Then $b=b_{2s-1-l, j}$. By the assumption and $F(a_{s-1-l, j}, \psi)=1$, $F(b_{2s-1-l, j}, \psi)=1$ holds. Therefore $F(b_{2s-1-(l+1), j}, \psi)=1$. Conversely, if we assume $F(b_{2s-1-(l+1), j}, \theta)=1$, then we deduce $F(a_{s-1-(l+1), j}, \theta)=1$. By induction on $l\leq k$, $F(a_{0, j}, \varphi)=1\leftrightarrow F(b_{s, j}, \varphi)=1$.
\end{lemproof}

\begin{lem}[$\mathrm{RCA}_0$]
$G$ is a valuation for $(A\cup B\cup X\cup C,R_{A\cup B\cup X\cup C})$.
\end{lem}

\begin{lemproof}
It is enough to check the condition for $\Box\varphi$. Fix $w\in A\cup B\cup X\cup C$ and $\varphi\in\mathrm{Fml}$. Without loss of generality, we may assume that $w=c_{j}\in C$ and $\mathrm{deg}(\varphi)=k$. 
\vspace{2mm}

First, we assume that $G(w, \Box\varphi)=1$ holds. Take $w^{\prime}\in wR$.
\begin{itemize}
    \item Case 1: $f(0)=j$. Then $wR=\{b_{0, j}, x_{0, j}\}$ and $F(b_{0, j}, \varphi)=1\land F(x_{0, j}, \varphi)=1$. So, $G(w^{\prime}, \varphi)=1$.
    \item Case 2: $0<\exists m\leq k(f(m)=j)$. Then $wR=\{a_{0, j}, b_{m, j}, x_{0, j}\}$ and $F(a_{0, j}, \varphi)=1\land F(b_{m, j}, \varphi)=1\land F(x_{0, j}, \varphi)=1$. So, $G(w^{\prime}, \varphi)=1$.
    \item Case 3: $\forall m\leq k(f(m)\not=j)$. If $\exists s>k(f(s)=j)$, then $wR=\{a_{0, j}, b_{s, j}, x_{0, j}\}$ and $F(a_{0, j}, \varphi)=1\land F(x_{0, j}, \varphi)=1$. By Lemma~\ref{A}, $F(a_{0, j}, \varphi)=1\leftrightarrow F(b_{s, j}, \varphi)=1$ holds. So, $G(w^{\prime}, \varphi)=1$. If $\forall s>k(f(s)\not=j)$, then $wR=\{a_{0, j}, x_{0, j}\}$. So, $G(w^{\prime}, \varphi)=1$.
\end{itemize}
\vspace{1mm}

Next, we assume that $G(w, \Box\varphi)=0$. 
\begin{itemize}
    \item Case 1: $f(0)=j$. Then $wR=\{b_{0, j}, x_{0, j}\}$ and $F(b_{0, j}, \varphi)=0\lor F(x_{0, j}, \varphi)=0$. So, there exists $w^{\prime}\in wR$ such that $G(w^{\prime}, \varphi)=0$.
    \item Case 2: $0<\exists m\leq k(f(m)=j)$. Then $wR=\{a_{0, j}, b_{m, j}, x_{0, j}\}$ and $F(a_{0, j}, \varphi)=0\lor F(b_{m, j}, \varphi)=0\lor F(x_{0, j}, \varphi)=0$. So, there exists $w^{\prime}\in wR$ such that $G(w^{\prime}, \varphi)=0$.
    \item Case 3: $\forall m\leq k(f(m)\not=j)$. If $\exists s>k(f(s)=j)$, then $wR=\{a_{0, j}, b_{s, j}, x_{0, j}\}$ and $F(a_{0, j}, \varphi)=0\lor F(x_{0, j}, \varphi)=0$. So, there exists $w^{\prime}\in wR$ such that $G(w^{\prime}, \varphi)=0$. If $\forall s>k(f(s)\not=j)$, then $wR=\{a_{0, j}, x_{0, j}\}$. So, there exists $w^{\prime}\in wR$ such that $G(w^{\prime}, \varphi)=0$.
\end{itemize}
\end{lemproof}

Next, we define the function $H$ by induction on the length of the formula. For all $w\in A\cup B\cup X\cup C\cup D$ and $p\in\mathrm{Prop}$, $H(w, p)\coloneq v(w, p)$. Next, fix $w\in A\cup B\cup X\cup C\cup D$ and $\varphi\in\mathrm{Fml}$. Assume that for all $w\in A\cup B\cup X\cup C\cup D$, $H(w, \varphi)$ is defined. Then
\begin{equation*}
\begin{split}
H(w, \Box\varphi)=1\Leftrightarrow\,(w\notin D\rightarrow G(w, \Box\varphi)=1)\land(w=d_{j}\in D\rightarrow H(d_{j+1}, \varphi)=1\land G(c_{j}, \varphi)=1).
\end{split}
\end{equation*}

Clearly, $H$ is an extension of $G$, and we can define such an $H$ using $\mathrm{I}\Sigma^0_1$.

\begin{lem}[$\mathrm{RCA}_0$]\label{H}
$H$ is a valuation for $(A\cup B\cup X\cup C\cup D,R_{A\cup B\cup X\cup C\cup D})$.
\end{lem}

\begin{lemproof}
It is enough to check the condition for $\Box\varphi$. Fix $w\in A\cup B\cup X\cup C\cup D$ and $\varphi\in\mathrm{Fml}$. Without loss of generality, we may assume that $w=d_{j}\in D$. 
\vspace{2mm}

First, we assume that $H(w, \Box\varphi)=1$ holds. Take $w^{\prime}\in wR=\{d_{j+1}, c_{j}\}$. By definition, $H(w^{\prime}, \varphi)=1$. Next, we assume that $H(w, \Box\varphi)=0$. By definition, there exists $w^{\prime}\in wR$ such that $H(w^{\prime}, \varphi)=0$. 
\end{lemproof}

Next, we define the function $V$ by induction on the length of the formula. For all $w\in W$ and $p\in\mathrm{Prop}$, $V(w, p)\coloneq v(w, p)$. Next, fix $w\in W$ and $\varphi\in\mathrm{Fml}$. Then
\begin{equation*}
\begin{split}
V(w,\Box\varphi)=1\Leftrightarrow\,(w\not=*\rightarrow H(w, \Box\varphi)=1)\land(w=*\rightarrow H(d_{0}, \varphi)=1).
\end{split}
\end{equation*}

Clearly, $V$ is an extension of $H$, and we can define such a $V$ using $\mathrm{I}\Sigma^0_1$. 

\begin{lem}[$\mathrm{RCA}_0$]
$V$ is a valuation for $(W,R)$.
\end{lem}

\begin{lemproof}
By Lemma~\ref{H}.
\end{lemproof}

Then we define $V^{\prime}:W^{\prime}\times\mathrm{Fml}\rightarrow2$ by $V^{\prime}=V\restriction(W^{\prime}\times\mathrm{Fml})$, and define two models by $M=(W, R, V)$ and $M^{\prime}=(W^{\prime}, R^{\prime}, V^{\prime})$. Next, we show that for all $j$, $(M, c_{j})\leftrightsquigarrow (M^{\prime}, c_{j})$ and $(M, d_{j})\leftrightsquigarrow (M^{\prime}, d_{j})$ hold. Clearly, for every $i$ and $ j$, $(M, a_{i, j})\leftrightsquigarrow (M^{\prime}, a_{i, j})$ and $(M, b_{i, j})\leftrightsquigarrow (M^{\prime}, b_{i, j})$ hold. 

\begin{lem}[$\mathrm{RCA}_0$]\label{1or2}
Fix $j\in\mathbb{N}$. Let $\varphi\in\mathrm{Fml}$ with $\mathrm{deg}(\varphi)=k$. Then the following statements hold: 
\begin{enumerate}
    \item If $\forall m\leq k(f(m)\not=j)$ holds, then $V(x_{0, j}, \varphi)=1\leftrightarrow V(a_{0, j}, \varphi)=1$,
    \item If $\exists m\leq k(f(m)=j)$ holds, then $\exists l\leq k((f(l)=j)\land(V(x_{0, j}, \varphi)=1\leftrightarrow V(b_{l, j}, \varphi)=1))$.
\end{enumerate}
%Moreover, we can find whether (1) is true or not recursively. 
%If (1) is not true, then we can find (the least) $l$ such that (2) is true recursively.
\end{lem}
\begin{lemproof}
(1): We assume that $\forall m\leq k(f(m)\not=j)$ holds. Then for all $s<k$, $a_{s, j}Ra_{s+1, j}$. 

We show that the following statement by induction on the length of the formula $\psi$:
\begin{align*}
\forall s\leq k\forall\psi\in\mathrm{Fml}(\mathrm{deg}(\psi)=s\rightarrow(V(x_{k-s, j}, \psi)=1\leftrightarrow V(a_{k-s, j}, \psi)=1)).
\end{align*}
By $\mathrm{I}\Sigma^0_1$, we can show it.
\begin{itemize}
    \item \textit{Base case}: For all $p\in\mathrm{Prop}$, $V(x_{k, j}, p)=1\leftrightarrow V(a_{k, j}, p)=1$ holds by the definition of $v$. 
    \item \textit{Induction step}: Fix $s<k$, and assume that for all $\psi$ with $\mathrm{deg}(\psi)=s$, $V(x_{k-s, j}, \psi)=1\leftrightarrow V(a_{k-s, j}, \psi)=1$ holds. Then take $\theta\equiv\Box\psi$. Assume that $V(a_{k-(s+1), j}, \theta)=1$. Take $x\in x_{k-(s+1), j}R$, then $x=x_{k-s, j}$. By the assumption and $V(a_{k-s, j}, \psi)=1$, $V(x_{k-s, j}, \psi)=1$ holds. Therefore $V(x_{k-(s+1), j}, \theta)=1$ holds. Conversely, if we assume $V(x_{k-(s+1), j}, \theta)=1$, then we deduce $V(a_{k-(s+1), j}, \theta)=1$. So, $V(x_{k-(s+1), j}, \theta)=1\leftrightarrow V(a_{k-(s+1), j}, \theta)=1$ holds. Therefore $V(x_{0, j}, \varphi)=1\leftrightarrow V(a_{0, j}, \varphi)=1$ holds. 
\end{itemize}

(2): We assume that $\exists m\leq k(f(m)=j)$ holds. Then we can search for $l\leq k$ such that $f(l)=j$. Then for all $s$, $b_{l+s, j}Rb_{l+(s+1), j}$. By a similar argument to that used for case (1), we can show that $V(x_{0, j}, \varphi)=1\leftrightarrow V(b_{l, j}, \varphi)=1$.
\end{lemproof}

\begin{lem}[$\mathrm{RCA}_0$]\label{C}
For all $j$, $(M, c_{j})\leftrightsquigarrow (M^{\prime}, c_{j})$.
\end{lem}

\begin{lemproof}
Fix $j\in\mathbb{N}$. We show that for all $\varphi\in\mathrm{Fml}$, $V(c_{j}, \varphi)=1\leftrightarrow V^{\prime}(c_{j}, \varphi)=1$ holds by induction on the length of $\varphi$. 
\vspace{2mm}

\textit{Base case}: For all $p\in\mathrm{Prop}$, $V(c_{j}, p)=1\leftrightarrow V^{\prime}(c_{j}, p)=1$ holds by the definition of $V$. 
\vspace{2mm}

\textit{Induction step}: Fix $s$ and assume that for all $\psi$ with $\mathrm{deg}(\psi)=s$, $V(c_{j}, \psi)=1\leftrightarrow V^{\prime}(c_{j}, \psi)=1$. Fix $\theta\equiv\Box\psi$. We assume that $V(c_{j}, \Box\psi)=1$ holds. Then by the definition of $R^{\prime}$ and the assumption, $V^{\prime}(c_{j}, \Box\psi)=1$ holds. Next, we assume that $V(c_{j}, \Box\psi)=0$ holds. Then there exists $w\in c_{j}R$ such that $V(w, \psi)=0$ holds. We may assume that $w=x_{0, j}$. We consider the following cases.
\begin{itemize}
    \item Case 1: $\forall m\leq k(f(m)\not=j)$ holds. Then $a_{0, j}\in c_{j}R^{\prime}$, and by lemma~\ref{1or2}, $V^{\prime}(a_{0, j}, \psi)=0$. So, $V^{\prime}(c_{j}, \Box\psi)=0$ holds.
    \item Case 2: $\exists l\leq k(f(l)=j)$ holds. Then $b_{l, j}\in c_{j}R^{\prime}$, and by lemma~\ref{1or2}, $V^{\prime}(b_{l, j}, \psi)=0$. So, $V^{\prime}(c_{j}, \Box\psi)=0$ holds.
\end{itemize}
By induction, for all $\varphi\in\mathrm{Fml}$, $V(c_{j}, \varphi)=1\leftrightarrow V^{\prime}(c_{j}, \varphi)=1$ holds. 
\end{lemproof}

\begin{lem}[$\mathrm{RCA}_0$]\label{D}
For all $j$, $(M, d_{j})\leftrightsquigarrow (M^{\prime}, d_{j})$.
\end{lem}

\begin{lemproof}
We show that for all $\varphi\in\mathrm{Fml}$, $\forall j(V(d_{j}, \varphi)=1\leftrightarrow V^{\prime}(d_{j}, \varphi)=1)$ by induction on the length of $\varphi$. By $\mathrm{I}\Sigma^0_1$, we can show it.
\vspace{2mm}

\textit{Base case}: For all $p\in\mathrm{Prop}$ and $j$, $V(d_{j}, p)=1\leftrightarrow V^{\prime}(d_{j}, p)=1$ by the definition of $V$. 
\vspace{2mm}

\textit{Induction step}: Fix $s$ and assume that for all $\psi\in\mathrm{Fml}$ with $\mathrm{deg}(\psi)=s$ and all $j$, $V(d_{j}, \psi)=1\leftrightarrow V^{\prime}(d_{j}, \psi)=1$ holds. Take $\theta=\Box\psi$ with $\mathrm{deg}(\psi)=s$ and $j$. We assume that $V(d_{j}, \Box\psi)=0$. Then there exists $w\in d_{j}R$ such that $V(w, \psi)=0$. $w$ is either $c_{j}$ or $d_{j+1}$. If $w=c_{j}$, then $V^{\prime}(w, \psi)=0$ by Lemma~\ref{C}. If $w=d_{j+1}$, then $V^{\prime}(w, \psi)=0$ by the assumption. So, $V^{\prime}(d_{j}, \Box\psi)=0$. In a similar way, we can show that if $V^{\prime}(d_{j}, \Box\psi)=0$, then $V(d_{j}, \Box\psi)=0$.
\vspace{2mm}

By induction, for all $\varphi\in\mathrm{Fml}$, $\forall j(V(d_{j}, \varphi)=1\leftrightarrow V^{\prime}(d_{j}, \varphi)=1)$ holds.
\end{lemproof}

By Lemma~\ref{D}, $(M, *)\leftrightsquigarrow (M^{\prime}, *)$ holds. Then, by the Hennessy-Milner theorem, $(M, *)\,\underline{\xleftrightarrow{}}\,(M^{\prime}, *)$, that is, there exists a bisimulation $Z\subseteq W\times W^{\prime}$ over $M$ and $M^{\prime}$ such that $*Z*$. Next, we show that for all $j$, $d_{j}Zd_{j}$ holds.  Then we find that for all $j$, $c_{j}Zc_{j}$ holds.

\begin{lem}[$\mathrm{RCA}_0$]\label{dzdczc}
For all $j$, $d_{j}Zd_{j}$ and $c_{j}Zc_{j}$ hold.
\end{lem}

\begin{lemproof}
We show that for all $j$, $d_{j}Zd_{j}$ by induction on $j$. By $*R=\{d_{0}\}$ and $*Z*$, $d_{0}Zd_{0}$ holds. Fix $j$ and we assume that $d_{j}Zd_{j}$. By the forth condition of the bisimulation $Z$, there exists $w\in d_{j}R$ such that $wZd_{j+1}$. Then, by $V(d_{j+1}, r_{j+1})=V^{\prime}(d_{j+1}, r_{j+1})=1$, $V(w, r_{j+1})=1$. So, $w=d_{j+1}$, that is, $d_{j+1}Zd_{j+1}$. By induction, for all $j$, $d_{j}Zd_{j}$. 
\vspace{2mm}

Next, we show that for all $j$, $c_{j}Zc_{j}$. Fix $j$. Then $d_{j}Zd_{j}$ holds, so, by the forth condition of the bisimulation $Z$, there exists $w\in d_{j}R$ such that $wZc_{j}$. Then, by $V(c_{j}, q_{j})=V^{\prime}(c_{j}, q_{j})=1$, $V(w, q_{j})=1$. So, $w=c_{j}$, that is, $c_{j}Zc_{j}$. 
\end{lemproof}

Then we define $J=\{j\in\mathbb{N}\mid\neg(x_{0, j}Za_{0, j})\}$. We show that $J$ is the range of $f$. 
\begin{itemize}
    \item Assume that $j\in J$. Then $\neg(x_{0, j}Za_{0, j})$ holds. By the bisimulation $Z$ and Lemma~\ref{dzdczc}, there exists $w\in c_{j}R^{\prime}$ such that $x_{0, j}Zw$. Note that $x_{0, j}\not\in c_{j}R^{\prime}$ and $w\not=a_{0, j}$. So, by the definition of $R^{\prime}\,(\text{that is}, R)$, there exists $m$ such that $w=b_{m, j}$, $f(m)=j$.
    \item We assume that $j\not\in J$. Then $x_{0, j}Za_{0, j}$ holds. To deduce a contradiction, we assume that there exists $m$ such that $f(m)=j$. $m\not=0$, otherwise $V(a_{0, j}, p_{0, j})=0$ and $V(x_{0, j}, p_{0, j})=1$. So, $m>0$. In fact, we can show that $x_{m-1, j}Za_{m-1, j}$ holds by $\Sigma^0_0$-induction. By the bisimulation $Z$ and $x_{m-1, j}Rx_{m, j}$, there exists $w\in a_{m-1, j}R^{\prime}$ such that $x_{m, j}Za_{m, j}$. However, $a_{m-1, j}R^{\prime}=\emptyset$. It is a contradiction. Therefore, for all $m$, $f(m)\not=j$. So, $\exists m(f(m)=j)\leftrightarrow j\in J$.
\end{itemize}
With the above, the proof is complete. 
\end{proof}

\section{The van Benthem characterization theorem in second-order arithmetic}

In this section, we study the strength of the van Benthem characterization theorem in second-order arithmetic. We consider three forms of the theorem: the semantic form, the syntactic form, and the hybrid form. We show that the semantic form is provable in $\mathrm{RCA}_0$, the syntactic form is provable in $\mathrm{PRA}$, and the hybrid form is equivalent to the lightface version of $\mathrm{WKL}$ over $\mathrm{RCA}_0$. %To prove this, we formalize Otto’s proof of the van Benthem characterization theorem~\cite{Otto2004b,MR2258711} in $\mathrm{RCA}_0$.

%we show that the van Benthem characterization theorem is provable in $\mathrm{WKL}_0$. Furthermore, we demonstrate that it is also provable in $\mathrm{PRA}$. 
First, we define the provability predicate for modal logic.
\begin{comment}
\begin{prop}[$\mathrm{RCA}_0$]\label{uniform}
There exists a function $f:\mathcal{F}\times\mathrm{Fml}\rightarrow\mathrm{Fml}$ such that for all finite function $\sigma:\subseteq\mathrm{Prop}\rightarrow\mathrm{Fml}$, $f(\sigma, \cdot)$ is the \emph{uniform substitution for $\sigma$}, i.e.,
\begin{itemize}
    \item $f(\sigma,p)=\sigma(p)$ if $p\in\mathrm{Prop}$ and $p\in\mathrm{dom}(\sigma)$,
    \item $f(\sigma,\alpha)=\alpha$ if $\alpha\in\mathrm{C}$,
    \item $f(\sigma,\bot)=\bot$,
    \item $f(\sigma,\varphi\rightarrow\psi)=f(\sigma,\varphi)\rightarrow f(\sigma,\psi)$,
    \item $f(\sigma,\Box\varphi)=\Box f(\sigma,\varphi)$.
\end{itemize}
\end{prop}

\begin{proof}
We define $f:\mathcal{F}\times\mathrm{Fml}\rightarrow\mathrm{Fml}$ recursively: for all $p\in\mathrm{Prop}\cup\mathrm{C}$ and $\sigma\in\mathcal{F}$, $f(\sigma, p)=\sigma(p)$ if $p\in\mathrm{Prop}\land p\in\mathrm{dom}(\sigma)$, otherwise $f(\sigma, p)=p$. This is possible using $\mathrm{I}\mathrm{\Sigma}^0_1$.
\end{proof}

Henceforth, for all $\varphi\in\mathrm{Fml}$ and $\sigma\in\mathcal{F}$, we denote $f(\sigma,\varphi)$ as $\overline{\sigma}(\varphi)$. Note that for all $\sigma\in\mathrm{Seq}$, we denote the length of $\sigma$ as $lh(\sigma)$.
We now formalize the notion of provability.
\end{comment}
\begin{dfn}[$\mathrm{RCA}_0$]
Let $\mathrm{\mathrm{\Gamma}}\subset\mathrm{Fml}$. We define the following predicates:
\begin{equation*}
\begin{split}
\mathrm{Prf}_{\mathbf{K}}(\mathrm{\mathrm{\Gamma}}, p)\equiv p&\in\mathrm{Seq}\land\forall k(k<lh(p)\rightarrow p(k)\in\mathrm{Fml})\\&\land\forall k\Bigl(k<lh(p)\rightarrow\Bigl(p(k)\in\mathrm{\Gamma}\lor p(k)\in\mathrm{Axm}\\&\lor(\exists i<k\exists j<k)(p(i)=p(j)\rightarrow p(k))\lor(\exists i<k)(p(k)=\Box p(i))\\&\lor(\exists i<k)(\exists\sigma:\{q<p(i)\mid q\in\mathrm{Prop}\}\rightarrow\{\varphi<p(k)\mid\varphi\in\mathrm{Fml}\})\\&(p(k)=\overline{\sigma}(p(i)))\Bigr)\Bigr),
\end{split}
\end{equation*}
\begin{equation*}
\begin{split}
\mathrm{Pbl}_{\mathbf{K}}(\mathrm{\mathrm{\Gamma}}, \varphi)\equiv&\exists\psi_1,\dots,\exists\psi_n\in\mathrm{\mathrm{\Gamma}}\exists p\\&\bigl(\mathrm{Prf}_{\mathbf{K}}(\emptyset, p)\land(\exists i<lh(p))(p(i)=(\psi_1\land\dots\land\psi_n\rightarrow\varphi))\bigr), 
\end{split}
\end{equation*}
\end{dfn}
where $\overline{\sigma}$ is the uniform substitution for a finite function $\sigma:\subseteq\mathrm{Prop}\rightarrow\mathrm{Fml}$ (See~\cite{10.1007/978-3-031-95908-0_32}). 

%\begin{dfn}[$\mathrm{RCA}_0$]
Let $\mathrm{\mathrm{\Gamma}}\subset\mathrm{Fml}$. $\mathrm{\mathrm{\Gamma}}$ is \emph{consistent} if $\neg\mathrm{Pbl}_{\mathbf{K}}(\mathrm{\mathrm{\Gamma}}, \bot)$ holds.
\begin{comment}
\begin{itemize}
    \item $\mathrm{\mathrm{\Gamma}}$ is \emph{consistent} if $\neg\mathrm{Pbl}_{\mathbf{K}}(\mathrm{\mathrm{\Gamma}}, \bot)$ holds,
    \item $\mathrm{\mathrm{\Gamma}}$ is \emph{closed under deduction} if $\forall\varphi\in\mathrm{Fml}(\mathrm{Pbl}_{\mathbf{K}}(\mathrm{\mathrm{\Gamma}}, \varphi)\rightarrow\varphi\in\mathrm{\Gamma})$ holds,
    \item $\mathrm{\mathrm{\Gamma}}$ is \emph{complete} if $\forall\varphi\in\mathrm{Fml}(\varphi\in \mathrm{\mathrm{\Gamma}}\lor\varphi\not\in\mathrm{\mathrm{\Gamma}})$ holds.
%    \item $\mathrm{\mathrm{\Gamma}}$ is \emph{maximally $\mathbf{K}\mathrm{\mathrm{\Sigma}}$-consistent} if $\mathrm{\mathrm{\Gamma}}$ is $\mathbf{K}\mathrm{\mathrm{\Sigma}}$-consistent, closed under deduction, and complete.
\end{itemize}
%\end{dfn}
\end{comment}

\begin{dfn}[$\mathrm{RCA}_0$]
Let $\mathcal{L}_{\mathcal{FO}}$ be the first-order language (with equality) which has unary predicates $P_0, P_1, \dots$ corresponding to the proposition letters $p_0, p_1, \dots$, and a binary relation symbol $r$.
\end{dfn}

As in the case of modal logic, we can show in $\mathrm{RCA}_0$ that there exist sets $\mathrm{Trm}_{\mathcal{FO}}$, $\mathrm{Fml}_{\mathcal{FO}}$, $\mathrm{Snt}_{\mathcal{FO}}$, and $\mathrm{Axm}_{\mathcal{FO}}$ consisting of all $\mathcal{L}_{\mathcal{FO}}$-terms, $\mathcal{L}_{\mathcal{FO}}$-formulas, $\mathcal{L}_{\mathcal{FO}}$-sentences, and axioms of first-order logic in $\mathcal{L}_{\mathcal{FO}}$, respectively. Moreover, we can define the provability predicate of first-order logic in $\mathrm{RCA}_0$. We assume that the only rule of inference is modus ponens (see~\cite{MR2517689}).

\begin{dfn}[$\mathrm{RCA}_0$~\cite{MR2517689}]
Let $X\subset\mathrm{Fml}_\mathcal{FO}$. We define the following formulas:
\begin{equation*}
\begin{split}
\mathrm{Prf}_{\mathcal{FO}}(X, p)\equiv p&\in\mathrm{Seq}\land\forall k(k<lh(p)\rightarrow p(k)\in\mathrm{Fml}_\mathcal{FO})\\&\land\forall k\Bigl(k<lh(p)\rightarrow\Bigl(p(k)\in X\lor p(k)\in\mathrm{Axm}_\mathcal{FO}\\&\lor(\exists i<k\exists j<k)(p(i)=p(j)\rightarrow p(k))\Bigr)\Bigr),
\end{split}
\end{equation*}
\begin{equation*}
\begin{split}
\mathrm{Pbl}_{\mathcal{FO}}(X, \varphi)\equiv&\exists\psi_1,\dots,\psi_n\in X\exists p\\&\bigl(\mathrm{Prf}_{\mathcal{FO}}(X, p)\land(\exists i<lh(p))(p(i)=(\psi_1\land\dots\land\psi_n\rightarrow\varphi))\bigr).
\end{split}
\end{equation*}
\end{dfn}

Furthermore, as in the case of modal logic, we can use $\mathrm{RCA}_0$ to describe the consistency of the system and the satisfiability of the model in first-order logic $\mathcal{FO}$.

\begin{dfn}[$\mathrm{RCA}_0$~\cite{MR2517689}]
Let $M_{\star}$ be a set. Let $T_M$ and $S_M$ be respectively the sets of closed terms and sentences of the expanded language $\mathcal{L}^{M}_{\mathcal{FO}}=\mathcal{L}_{\mathcal{FO}}\cup\{\overline{m}\mid m\in M_{\star}\}$ with new constant symbols $\overline{m}$ for each element $m\in M_{\star}$. 

$M=(M_{\star}, V)$ is a standard model on $\mathcal{L}_{\mathcal{FO}}$ if the function $V: T_{M}\cup S_{M}\rightarrow M_{\star}\cup\{0, 1\}$ satisfies the following conditions:
\begin{itemize}
    \item $V(t)\in M_{\star}$ for any $t\in T_M$,
    \item $V(\varphi)\in\{0, 1\}$ for any $\varphi\in S_M$,
    \item $\bigl(\forall i\leq n(V(t_i)=V(t^{\prime}_i))\rightarrow V(R(t_1,\dots,t_n))=V(R(t^{\prime}_1,\dots,t^{\prime}_n))\bigr)$, \\where $R$ is a relation symbol, for any $t_1,  t^{\prime}_1, \dots  t_n, t^{\prime}_n\in T_M$,
    \item $V(\bot)=0$,
    \item $V(\varphi\rightarrow\psi)=1-V(\varphi)(1-V(\psi))$ for any $\varphi,\psi\in S_M$,
    \item $V(\forall x\varphi(x))=1\iff\forall a\in M_{\star}(V(\varphi(\overline{a}))=1)$  for any $\varphi(\overline{a})\in S_M$.
\end{itemize}
\end{dfn}

Let $X\subseteq\mathrm{Snt}_\mathcal{FO}$. Then a standard model $M$ is  a model of $X$ on $\mathcal{L}_{\mathcal{FO}}$ if $\forall\varphi\in X(V(\varphi)=1)$ holds.

\begin{dfn}[$\mathrm{RCA}_0$]
Let $\varphi\in\mathrm{Snt}_{\mathcal{FO}}$ and $X\subseteq\mathrm{Snt}_{\mathcal{FO}}$. Let $M=(M_{\star}, V)$ be a standard model on $\mathcal{L}_{\mathcal{FO}}$. We write $M\vDash\varphi$ if $M$ satisfies $\varphi$, i.e., $V(\varphi)=1$, and we write $X\vDash\varphi$ if every model of $X$ on $\mathcal{L}_{\mathcal{FO}}$ satisfies $\varphi$.
\end{dfn}

\begin{thm}[The Soundness Theorem for $\mathcal{FO}$, $\mathrm{RCA}_0$ \cite{MR2517689}]
Let $\varphi\in\mathrm{Snt}_{\mathcal{FO}}$. If $\mathrm{Pbl}_{\mathcal{FO}}(\emptyset, \varphi)$ holds, then $\vDash\varphi$ holds.
\end{thm}

\begin{thm}[G\"odel's Completeness Theorem for $\mathcal{FO}$, $\mathrm{WKL}_0$ \cite{MR2517689}]
Let \\$\varphi\in\mathrm{Snt}_{\mathcal{FO}}$. If $X\subseteq\mathrm{Snt}_{\mathcal{FO}}$ is consistent, then $\mathrm{Pbl}_{\mathcal{FO}}(X, \varphi)$ holds if and only if $X\vDash\varphi$ holds.
\end{thm}

\begin{cor}[Weak Completeness Theorem for $\mathcal{FO}$, $\mathrm{WKL}_0$]
Let $\psi\in\mathrm{Snt}_{\mathcal{FO}}$ and let $T\subseteq\mathrm{Snt}_{\mathcal{FO}}$ be a finite set. If $T$ is consistent, then $\mathrm{Pbl}_{\mathcal{FO}}\bigl(T, \psi)\bigr)$ holds if and only if $T\vDash\psi$ holds.
\end{cor}

\begin{proof}
By G\"odel's Completeness Theorem for $\mathcal{FO}$.
\end{proof}

The following is essentially a first-order reformulation of the modal satisfaction definition.

\begin{dfn}[$\mathrm{RCA}_0$~\cite{MR1837791}]
Let $x$ be a first-order variable. The \emph{standard translation} $\mathrm{ST}_{x}$ taking modal formulas to $\mathcal{L}_{\mathcal{FO}}$-formulas is defined as follows:
\begin{itemize}
    \item $\mathrm{ST}_{x}(p_{n})\equiv P_{n}(x)$,
    \item $\mathrm{ST}_{x}(\bot)\equiv\bot$,
    \item $\mathrm{ST}_{x}(\varphi\rightarrow\psi)\equiv \mathrm{ST}_{x}(\varphi)\rightarrow \mathrm{ST}_{x}(\psi)$,
    \item $\mathrm{ST}_{x}(\Box\varphi)\equiv\forall y(r(x, y)\rightarrow \mathrm{ST}_{y}(\varphi))$   ($y$ is fresh).
\end{itemize}
\end{dfn}

\begin{lem}[$\mathrm{RCA}_0$~\cite{10.1007/978-3-031-95908-0_32}]\label{frightk}
Let $\varphi\in\mathrm{Fml}$. For all $M=(M_{\star}, V)$ which is a standard model on $\mathcal{L}_{\mathcal{FO}}$, there exists a Kripke model $M^{\star}$ such that
\begin{enumerate}
    \item for all $w\in M$, $M^{\star}, w\Vdash\varphi\iff M\vDash \mathrm{ST}_{w}(\varphi)$, and
    \item $M^{\star}\Vdash\varphi\iff M\vDash\forall x\mathrm{ST}_{x}(\varphi)$.
\end{enumerate}
\end{lem}

In this sense, we identify the Kripke model with the standard model on $\mathcal{L}_{\mathcal{FO}}$. For a Kripke model $M=(W, R, V)$, a world $w\in W$, and a modal formula $\varphi\in\mathrm{Fml}$, we treat $V(w, \varphi)$ and $V(\mathrm{ST}_{w}(\varphi))$ as synonymous.

\begin{dfn}[$\mathrm{RCA}_0$~\cite{MR1837791}]
Let $\varphi(x)\in\mathrm{Fml}_{\mathcal{FO}}$ which has no free number variables except $x$. $\varphi(x)$ is \emph{bisimulation invariant} if for all models $M$ and $N$, all worlds $w\in M$, $v\in N$, and all bisimulations $Z$ between $M$ and $N$ such that $wZv$, we have $M\vDash\varphi(w)$ if and only if $N\vDash\varphi(v)$.
\vspace{1mm}

$\varphi(x)$ is \emph{bisimulation invariant for finite models} if for all finite models $M$ and $N$, all worlds $w\in M$, $v\in N$, and all bisimulations $Z$ between $M$ and $N$ such that $wZv$, we have $M\vDash\varphi(w)$ if and only if $N\vDash\varphi(v)$.
\end{dfn}

\subsection{Semantic form of the van Benthem characterization theorem}

In this subsection, we show that the semantic form of the van Benthem characterization theorem is provable in $\mathrm{RCA}_0$. To prove this, we formalize Otto’s proof of the van Benthem characterization theorem~\cite{Otto2004b,MR2258711} in $\mathrm{RCA}_0$.

First, we consider the two forms of the van Benthem characterization theorem: the semantic form and the hybrid form.

\begin{state}[The van Benthem characterization theorem (the semantic form)]
Let $\varphi(x)\in\mathrm{Fml}_{\mathcal{FO}}$ which has no free number variables except $x$. Then the following are equivalent:
\begin{enumerate}
    \item $\varphi(x)$ is bisimulation invariant,
    \item there exists $\psi\in\mathrm{Fml}$ such that $\vDash\forall x(\varphi(x)\leftrightarrow\mathrm{ST}_{x}(\psi))$, that is, every standard model $M$ on $\mathcal{L}_{\mathcal{FO}}$ satisfies $\varphi$.
\end{enumerate}
\end{state}

\begin{state}[The van Benthem characterization theorem (the hybrid form)]\label{vanB}
Let $\varphi(x)\in\mathrm{Fml}_{\mathcal{FO}}$ which has no free number variables except $x$. Then the following are equivalent:
\begin{enumerate}
    \item $\varphi(x)$ is bisimulation invariant,
    \item there exists $\psi\in\mathrm{Fml}$ such that $\mathrm{Pbl}_{\mathcal{FO}}\bigl(\emptyset, \forall x(\varphi(x)\leftrightarrow\mathrm{ST}_{x}(\psi))\bigr)$ holds.
\end{enumerate}
\end{state}

For the two statements, we show that $(2)$ implies $(1)$ in $\mathrm{RCA}_0$.

\begin{prop}[$\mathrm{RCA}_0$~\cite{MR1837791}]\label{2im1vB1}
For every $\varphi(x)\in\mathrm{Fml}_{\mathcal{FO}}$ which has no free number variables except $x$, if  there exists $\psi\in\mathrm{Fml}$ such that $\mathrm{Pbl}_{\mathcal{FO}}\bigl(\emptyset, \forall x(\varphi(x)\leftrightarrow\mathrm{ST}_{x}(\psi))\bigr)$ holds, then $\varphi(x)$ is bisimulation invariant.
\end{prop}

\begin{proof}
Fix two models $M=(W, R, V)$ and $M^{\prime}=(W^{\prime}, R^{\prime}, V^{\prime})$. Let $w\in M$ and $w^{\prime}\in M^{\prime}$ and let $Z\subseteq W\times W^{\prime}$ be the bisimulation such that $wZw^{\prime}$. By the assumption and the Soundness Theorem for first-order logic, there exists $\psi\in\mathrm{Fml}$ such that for every model $M^{\prime\prime}=(W^{\prime\prime}, R^{\prime\prime}, V^{\prime\prime})$ and every $w^{\prime\prime}\in W^{\prime\prime}$, $M^{\prime\prime}\vDash\varphi(w^{\prime\prime})\leftrightarrow\mathrm{ST}_{w^{\prime\prime}}(\psi)$. Thus it is enough to show the following by induction on the length of $\psi$:
\begin{align*}
\forall w\in W\forall w^{\prime}\in M^{\prime}(wZw^{\prime}\rightarrow(V(\mathrm{ST}_{w}(\psi))=1\leftrightarrow V^{\prime}(\mathrm{ST}_{w^{\prime}}(\psi))=1)).
\end{align*}

\textit{Base case}: Let $\psi\equiv p\in\mathrm{Prop}$. Then $\mathrm{ST}_{x}(p)\equiv P(x)$. Fix $w\in W$ and $w^{\prime}\in W^{\prime}$ with $wZw^{\prime}$. By the bisimulation $Z$, $V(w, p)=1$ if and only if $V(w^{\prime}, p)=1$. Thus $V(\mathrm{ST}_{w}(p))=1$ if and only if $V(\mathrm{ST}_{w^{\prime}}(p))=1$.  
\vspace{2mm}

\textit{Induction step}: Fix $\psi\in\mathrm{Fml}$. Assume that for all $w\in W$ and $w^{\prime}\in W^{\prime}$ with $wZw^{\prime}$, $V(\mathrm{ST}_{w}(\psi))=1\leftrightarrow V^{\prime}(\mathrm{ST}_{w^{\prime}}(\psi))=1$. Fix $w\in W$ and $w^{\prime}\in W^{\prime}$ with $wZw^{\prime}$. Then, for $\Box\psi$, $\mathrm{ST}_{x}(\Box\psi)\equiv\forall y(r(x, y)\rightarrow\mathrm{ST}_{y}(\psi))$. Assume that $V(\mathrm{ST}_{w}(\Box\psi))=1$. Then for all $v\in wR$, $V(\mathrm{ST}_{v}(\psi))=1$. Take $v^{\prime}\in w^{\prime}R^{\prime}$. Then by the bisimulation $Z$, there exists $v\in wR$ such that $vZv^{\prime}$. Thus, by the assumption, $V^{\prime}(\mathrm{ST}_{v^{\prime}}(\psi))=1$. Thus $V^{\prime}(\mathrm{ST}_{w^{\prime}}(\Box\psi))=1$. Conversely, in a similar way, $V^{\prime}(\mathrm{ST}_{w^{\prime}}(\Box\psi))=1$ implies  $V(\mathrm{ST}_{w}(\Box\psi))=1$. Therefore $V(\mathrm{ST}_{w}(\Box\psi))=1\leftrightarrow V^{\prime}(\mathrm{ST}_{w^{\prime}}(\Box\psi))=1$ holds. By induction, this completes the proof.
\end{proof}

\begin{cor}[$\mathrm{RCA}_0$]
For every $\varphi(x)\in\mathrm{Fml}_{\mathcal{FO}}$ which has no free number variables except $x$, if  there exists $\psi\in\mathrm{Fml}$ such that $\vDash\forall x(\varphi(x)\leftrightarrow\mathrm{ST}_{x}(\psi))$ holds, then $\varphi(x)$ is bisimulation invariant.
\end{cor}

\begin{proof}
By Proposition~\ref{2im1vB1}.
\end{proof}

To show the semantic form over $\mathrm{RCA}_0$, we define the following objects.

\begin{dfn}[$\mathrm{RCA}_0$]
Fix $l\in\mathbb{N}$. Then we denote the predicates $r^{l}(x, y)$ and $r^{\leq l}(x, y)$ by the following:
\begin{itemize}
    \item $r^{0}(x, y)\equiv(x=y)$,
    \item $r^{m}(x, y)\equiv\exists z(r^{m-1}(x,z)\land r(z, y))$ $(m\leq l)$,
    \item $r^{\leq l}(x, y) \equiv \bigvee_{k=0}^{l} r^{k}(x, y)$.
\end{itemize} 

Let $z$ be the number variable and $\varphi(\vec{x})\in\mathrm{Fml}_{\mathcal{FO}}$ that does not contain the variable $z$. Then we define $\varphi^{\leq l}_{z}(\vec{x})\equiv(\varphi(\vec{x}))^{\leq l}_{z}$ as follows:
\begin{itemize}
    \item $(r(x, y))^{\leq l}_{z}\equiv r^{\leq l}(z, x)\land r^{\leq l}(z, y)\land r(x, y)$,
    \item $(P_{i}(x))^{\leq l}_{z}\equiv r^{\leq l}(z, x)\land P_{i}(x)$ (for all $i$),
    \item $(\varphi\rightarrow\psi)^{\leq l}_{z}\equiv(\varphi)^{\leq l}_{z}\rightarrow(\psi)^{\leq l}_{z}$,
    \item $(\forall y\varphi(\vec{x}, y))^{\leq l}_{z}\equiv\forall y(r^{\leq l}(z, y)\rightarrow(\varphi(\vec{x}, y))^{\leq l}_{z})$,
%    \item $(\exists y\varphi(\vec{x}, y))^{\leq l}_{z}\equiv\exists y(r^{\leq l}(z, y)\land(\varphi(\vec{x}, y))^{\leq l}_{z})$.
\end{itemize}

Let $M=(W, R, V)$ be a model. We define $\pi:W^{<l}\rightarrow W$ as follows: $\pi(\sigma)\coloneq\sigma_{lh(\sigma)-1}$.
\end{dfn}

\begin{dfn}[$\mathrm{RCA}_0$~\cite{Otto2004b,MR2258711}]
Let $M=(W, R, V)$ be a model and $a\in W$ and $l>0$. The \emph{$l$-locally tree unravelling} of $M$ from $a\in W$ is a model $M^{l}[a]\coloneq(W^{l}[a], R^{l}[a], V^{l}[a])$ where
\begin{itemize}
    \item $W^{l}[a]\coloneq\{\sigma\in W^{< l}\mid\sigma_{0}=a\land\forall i<lh(\sigma)-1(\sigma_{i}R\sigma_{i+1})\},$
    \item $R^{l}[a]\coloneq\{(\sigma, \tau)\in W^{l}[a]\times W^{l}[a]\mid\sigma\preceq\tau\land (lh(\sigma)+1=lh(\tau))\land\sigma_{lh(\sigma)-1}R\tau_{lh(\tau)-1}\},$
    \item $V^{l}[a](\varphi(\sigma^{0}, \sigma^{1}, \dots, \sigma^{s}))\coloneq V(\varphi^{\leq l}_{a}(\pi(\sigma^{0}), \pi(\sigma^{1}), \dots, \pi(\sigma^{s})))$\\
\vspace{-3mm}

\text{for all $\varphi(\vec{x})\in\mathrm{Fml}_{\mathcal{FO}}$ and $\sigma^{0}, \sigma^{1}, \dots, \sigma^{s}\in W^{l}[a]$}.
\end{itemize}
\end{dfn} 

The $l$-locally tree unravelling of $M$ from $a$ is a tree model with root $a$ and it consists of all nodes reachable from $a$ on (directed, forward) $R$-paths of length up to $l$. %Note that we also identify the Kripke model $M^{l}[a]$ with the standard model on $\mathcal{L}_{\mathcal{FO}}$. Moreover, b
By the tree unravelling, we often identify elements of $M$ whose path length along $R$ from $a$ is less than $l$ with elements of $M^{l}[a]$.

\begin{center}
\begin{tikzpicture}[
    scale=0.7,
    >=Stealth,
    thick
]

    % ================= 左側の図: M =================
    \begin{scope}[shift={(0,0)}]
        % 頂点 a
        \node[circle, fill, inner sep=1.5pt, label=below:{\large $a$}] (A1) at (0,0) {};
        
        % 逆錐の塗りつぶしと境界線
        \fill[gray!10] (0,0) -- (-1.5,3) -- (1.5,3) -- cycle;
        \draw[thick] (-1.5,3) -- (0,0) -- (1.5,3);
        
        % 中央のラベル M
        \node[font=\Large] at (0,2) {$M$};
    \end{scope}

    % ================= 右側の図: M^l[a] =================
    \begin{scope}[shift={(5,0)}]
        % 頂点 a
        \node[circle, fill, inner sep=1.5pt, label=below:{\large $a$}] (A2) at (0,0) {};
        
        % 上部が閉じた逆錐の塗りつぶしと境界線
        \fill[gray!10] (0,0) -- (-1.2,2.4) -- (1.2,2.4) -- cycle;
        \draw[thick] (0,0) -- (-1.2,2.4) -- (1.2,2.4) -- cycle;
        
        % 中央のラベル M^l[a]
        \node[font=\large] at (0,1.6) {$M^l[a]$};
    \end{scope}

\end{tikzpicture}
\end{center}

\begin{dfn}[$\mathrm{RCA}_0$~\cite{Otto2004b,MR2258711}]
Let $\varphi(x)\in\mathrm{Fml}_{\mathcal{FO}}$ which has no free number variables except $x$ and let $M=(W, R, V)$ be a model and $a\in W$ and $l>0$. $\varphi(x)$ is \emph{$l$-local} if $M\vDash\varphi(a)$ holds if and only if $M^{l}[a]\vDash\varphi(a)$.
\end{dfn}

We introduce the following as a formalization of the finite version of the Ehrenfeucht-Fra\"{i}ss\'{e} game~\cite{MR126370, MR69236} within $\mathrm{RCA}_0$.

\begin{dfn}[Finite Ehrenfeucht-Fra\"{i}ss\'{e} game, $\mathrm{RCA}_0$]
Let $M=(W, R, V)$, $N=(W^{\prime}, R^{\prime}, V^{\prime})$ and $q, k>0$. Let $f_{\mathrm{I}}:(W\cup W^{\prime})^{< 2q}\rightarrow W\cup W^{\prime}$ and $f_{\mathrm{II}}:(W\cup W^{\prime})^{< 2q}\rightarrow W\cup W^{\prime}$. Then for $\sigma\in(W\cup W^{\prime})^{< 2q}$,
\begin{equation*}
\begin{split}
(f_{\mathrm{I}}\#f_{\mathrm{II}})\restriction2q=\sigma:\Leftrightarrow&\,(lh(\sigma)=2q)\land(\sigma_{0}=f_{\mathrm{I}}(\langle\rangle))\\&\land\forall i<q-1(\sigma_{2i}=f_{\mathrm{I}}(\sigma\restriction(2i)))\land \sigma_{2i+1}=f_{\mathrm{II}}((\sigma\restriction(2i+1)))\\&\land(\sigma_{0}\in W\rightarrow \forall i<q-1(\sigma_{2i}\in W\land\sigma_{2i+1}\in W^{\prime}))\\&\land(\sigma_{0}\in W^{\prime}\rightarrow \forall i<q-1(\sigma_{2i}\in W^{\prime}\land\sigma_{2i+1}\in W)).
\end{split}
\end{equation*}
\begin{itemize}
    \item $\mathrm{II}$ has a winning strategy in the $q$-round Ehrenfeucht-Fra\"{i}ss\'{e} game from $(M;N)$ if there exists $f_{\mathrm{II}}:(W\cup W^{\prime})^{< 2q}\rightarrow W\cup W^{\prime}$ such that for all $f_{\mathrm{I}}:(W\cup W^{\prime})^{< 2q}\rightarrow W\cup W^{\prime}$, there exists $\sigma\in(W\cup W^{\prime})^{< 2q}$ such that $(f_{\mathrm{I}}\#f_{\mathrm{II}})\restriction2q=\sigma$, and $E=\{\sigma_{2i}\mid i<q\}$ is isomorphic to $O=\{\sigma_{2i-1}\mid i<q\}$, that is, for all predicates $P_{i}(x)$ and $r(x, y)$, and all $\sigma_{2a}, \sigma_{2b}\in E$, $P_{i}(\sigma_{2a})\text{ holds}\iff P_{i}(\sigma_{2a-1})$ holds, and $r(\sigma_{2a}, \sigma_{2b})\text{ holds}\iff r(\sigma_{2a-1}, \sigma_{2b-1})$ holds. 
    \item $\mathrm{II}$ has a winning strategy in the $q$-round $k$-restricted Ehrenfeucht-Fra\"{i}ss\'{e} game from $(M;N)$ if there exists $f_{\mathrm{II}}:(W\cup W^{\prime})^{< 2q}\rightarrow W\cup W^{\prime}$ such that for all $f_{\mathrm{I}}:(W\cup W^{\prime})^{< 2q}\rightarrow W\cup W^{\prime}$, there exists $\sigma\in(W\cup W^{\prime})^{< 2q}$ such that $(f_{\mathrm{I}}\#f_{\mathrm{II}})\restriction2q=\sigma$, and for $k+1$ predicates $P_{0}(x),\dots,P_{k-1}(x)$ and $r(x, y)$, and all $\sigma_{2a}, \sigma_{2b}\in E$, $P_{i}(\sigma_{2a})$ holds $\iff P_{i}(\sigma_{2a-1})$ holds $(i<k)$ and $r(\sigma_{2a}, \sigma_{2b})$ holds $\iff r(\sigma_{2a-1}, \sigma_{2b-1})$ holds. 
\end{itemize}
\end{dfn}

\begin{dfn}[$\mathrm{RCA}_0$~\cite{MR1837791}]
Let $M=(W, R, V)$ and $M^{\prime}=(W^{\prime}, R^{\prime}, V^{\prime})$ be two models, $w\in W$, $w^{\prime}\in W^{\prime}$, and $l>0$. A sequence of binary relations $Z_{0}\subseteq\dots\subseteq Z_{l}\subseteq W\times W^{\prime}$ is called an $l$-bisimulation between $(M, w)$ and $(M^{\prime}, w^{\prime})$ if the following conditions are satisfied:
\begin{enumerate}
    \item $w Z_{0} w^{\prime}$.
    \item $\forall v\forall v^{\prime}(vZ_{l}v^{\prime}\rightarrow\forall p\in\mathrm{Prop} (V(v, p)=1\leftrightarrow V^{\prime}(v^{\prime}, p)=1))$.
    \item $\forall i<l\forall v\forall v^{\prime}\forall u(vZ_{i}v^{\prime}\land vRu\rightarrow\exists u^{\prime}(uZ_{i+1}u^{\prime}\land v^{\prime}R^{\prime}u^{\prime}))$ (the forth condition).
    \item $\forall i<l\forall v\forall v^{\prime}\forall u(vZ_{i}v^{\prime}\land v^{\prime}R^{\prime}u^{\prime}\rightarrow\exists u(uZ_{i+1}u^{\prime}\land vRu))$ (the back condition).
\end{enumerate}
$(M, w)$ and $(M^{\prime}, w^{\prime})$ are $l$-bisimilar (we denote $(M, w)\,\underline{\xleftrightarrow{}}^{l}\,(M^{\prime}, w^{\prime})$) if there exists a $l$-bisimulation $Z_{0}\subseteq\dots\subseteq Z_{l}\subseteq W\times W^{\prime}$ between $(M, w)$ and $(M^{\prime}, w^{\prime})$.
\vspace{1mm}

Let $k>0$. A sequence of binary relations $Z_{0}\subseteq\dots\subseteq Z_{l}\subseteq W\times W^{\prime}$ is called a $k$-restricted $l$-bisimulation between $(M, w)$ and $(M^{\prime}, w^{\prime})$ if the above conditions $(1)$, $(3)$, $(4)$ and the following condition $(2^{\prime})$ are satisfied:
\begin{align*}
(2^{\prime})\,\forall v\forall v^{\prime}(vZ_{l}v^{\prime}\rightarrow\forall p\in\{p_{0}, \dots, p_{k-1}\} (V(w, p)=1\leftrightarrow V^{\prime}(w^{\prime}, p)=1)).
\end{align*}

$(M, w)$ and $(M^{\prime}, w^{\prime})$ are $k$-restricted $l$-bisimilar (we denote $(M, w)\,\underline{\xleftrightarrow{}}^{k, l}_{*}\,(M^{\prime}, w^{\prime})$) if there exists a $k$-restricted $l$-bisimulation $Z_{0}\subseteq\dots\subseteq Z_{l}\subseteq W\times W^{\prime}$ between $(M, w)$ and $(M^{\prime}, w^{\prime})$.
\end{dfn}

\begin{dfn}[$\mathrm{RCA}_0$~\cite{MR1837791}]
Let $\varphi(x)\in\mathrm{Fml}_{\mathcal{FO}}$ which has no free number variables except $x$ and $l>0$. $\varphi(x)$ is \emph{$l$-bisimulation invariant} if for all models $M$ and $N$, all worlds $w\in M$, $v\in N$, and all $l$-bisimulations $Z$ between $M$ and $N$ such that $wZv$, we have $M\vDash\varphi(w)$ if and only if $N\vDash\varphi(v)$.
%$\varphi(x)$ is \emph{$k$-restricted $l$-bisimulation invariant} if for all models $M$ and $N$, all worlds $w\in M$, $v\in N$, and all $k$-restricted $l$-bisimulations $Z$ between $M$ and $N$ such that $wZv$, we have $M\vDash\varphi(w)$ if and only if $N\vDash\varphi(v)$.
\end{dfn}

To show the semantic form over $\mathrm{RCA}_0$, we show the following lemmas. For $\psi\in\mathrm{Fml}_{\mathcal{FO}}$ whose quantifier rank is $q$, we denote $\mathrm{Qr}(\psi)=q$.

\begin{lem}[$\mathrm{RCA}_0$~\cite{MR126370, MR69236}]\label{efgame}
Let $M=(W, R, V)$, $N=(W^{\prime}, R^{\prime}, V^{\prime})$, $a\in W\cap W^{\prime}$, and $q>0$. If $\mathrm{II}$ has a winning strategy in the $q$-round Ehrenfeucht-Fra\"{i}ss\'{e} game from $(M;N)$, then for all $\psi(x)\in\mathrm{Fml}_{\mathcal{FO}}$ which has no free number variables except $x$ and $\mathrm{Qr}(\psi)\leq q$, $M\vDash\psi(a)\Leftrightarrow N\vDash\psi(a)$.
\end{lem}

\begin{proof}
For all $b_{1},\dots,b_{k}\in W\cup W^{\prime}$, $c_{b_{1}}\coloneq f_{\mathrm{II}}(\langle{b_{1}}\rangle)$ and $c_{b_{i}}\coloneq f_{\mathrm{II}}(\langle{b_{1}, c_{b_{i}},\dots, b_{i-1}, c_{b_{i-1}}, b_{i}}\rangle)$ for all $i<k+1$ $(k\leq q)$. 
\vspace{2mm}

We show the following statement by induction on the length of a quantifier-free formula $\psi(x, x_{1},\dots,x_{k})$ which has no free number variables except $x, x_{1}, \dots, x_{k}$: for all $b_{1},\dots,b_{k}\in C$ $(k\leq q)$, $M\vDash\psi(a, b_{1},\dots, b_{k})\Leftrightarrow N\vDash\psi(a, c_{b_{1}},\dots, c_{b_{k}})$.
\begin{itemize}
    \item \textit{Base case}: For all $i$ and $b, d\in C$, clearly, $M\vDash P_{i}(b)\Leftrightarrow N\vDash P_{i}(c_{b})$ and $M\vDash r(b, d)\Leftrightarrow N\vDash r(c_{b}, c_{d})$. Of course, $M\vDash P_{i}(a)\Leftrightarrow N\vDash P_{i}(a)$ and $M\vDash r(a, d)\Leftrightarrow N\vDash r(a, c_{d})$ and $M\vDash r(b, a)\Leftrightarrow N\vDash r(c_{b}, a)$. 
    \item \textit{Induction step}: Assume that for all $b_{1},\dots,b_{k}\in C$, $\psi(x, x_{1},\dots,x_{k}),\theta(x, x_{1},\dots,x_{k})\in\mathrm{Fml}_{\mathcal{FO}}$ which has no free number variables except $x, x_{1}, \dots, x_{k}$, $M\vDash\psi(a, b_{1},\dots, b_{k})\Leftrightarrow N\vDash\psi(a, c_{b_{1}},\dots, c_{b_{k}})$ and $M\vDash\theta(a, b_{1},\dots, b_{k})\Leftrightarrow N\vDash\theta(a, c_{b_{1}},\dots, c_{b_{k}})$. Then $M\vDash\psi(a, b_{1},\dots, b_{k})\land\theta(a, b_{1},\dots, b_{k})\Leftrightarrow N\vDash\psi(a, c_{b_{1}},\dots, c_{b_{k}})\land\theta(a, c_{b_{1}},\dots, c_{b_{k}})$. Thus we showed the case of $\psi(x, x_{1},\dots,x_{k})\land\theta(x, x_{1},\dots,x_{k})$. 
\end{itemize}
By induction, we complete the proof.
\vspace{2mm}

Similarly, we may also show that for all $b_{1},\dots,b_{k}\in C^{\prime}$ $(k\leq q)$ and a quantifier-free formula $\psi(x, x_{1},\dots,x_{k})\in\mathrm{Fml}_{\mathcal{FO}}$ which has no free number variables except $x, x_{1}, \dots, x_{k}$, $M\vDash\psi(a, c_{b_{1}},\dots, c_{b_{k}})\Leftrightarrow N\vDash\psi(a, b_{1},\dots, b_{k})$. 
\vspace{2mm}

Next, we show the following statement by induction on the length of $\psi$ whose quantifier rank is less than $q$: $M\vDash\psi(a)\Leftrightarrow N\vDash\psi(a)$. 
\begin{itemize}
    \item \textit{Base case}: For the case of quantifier-free formulas, we show it by the above.
    \item \textit{Induction step}: Fix $\forall x_{1}\dots\forall x_{k}\psi(x, x_{1},\dots,x_{k})\in\mathrm{Fml}_{\mathcal{FO}}$ which has no free number variables except $x$ and $k\leq q$. We assume that for all $b_{1},\dots,b_{k}\in C$, $N\vDash\psi(a, c_{b_{1}},\dots,c_{b_{k}})$ if and only if $M\vDash\psi(a, b_{1},\dots, b_{k})$. Assume that $N\vDash\forall x_{1}\dots\forall x_{k}\psi(a, x_{1},\dots,x_{k})$. Take $b_{1},\dots,b_{k}\in C$. Then $N\vDash\psi(a, c_{b_{1}},\dots,c_{b_{k}})$. By the assumption, $M\vDash\psi(a, b_{1},\dots, b_{k})$. Thus $M\vDash\forall x_{1}\dots\forall x_{k}\psi(a, x_{1},\dots, x_{k})$. In a similar way, we show that $M\vDash\forall x_{1}\dots\forall x_{k}\psi(a, x_{1},\dots, x_{k})$ implies $N\vDash\forall x_{1}\dots\forall x_{k}\psi(a, x_{1},\dots,x_{k})$. 
\end{itemize}
By induction, we complete the proof.
\end{proof}

\begin{cor}[$\mathrm{RCA}_0$]\label{k-efgame}
Let $M=(W, R, V)$, $N=(W^{\prime}, R^{\prime}, V^{\prime})$, $a\in W\cap W^{\prime}$, and $q, k>0$. If $\mathrm{II}$ has a winning strategy in the $q$-round $k$-restricted Ehrenfeucht-Fra\"{i}ss\'{e} game from $(M;N)$, then for all $\psi(x)\in\mathrm{Fml}_{\mathcal{FO}}$ which has no free number variables except $x$, contains only $P_{0}(y), \dots, P_{k-1}(y)$, and $r(y, y^{\prime})$ as first-order predicates, and $\mathrm{Qr}(\psi)\leq q$, $M\vDash\psi(a)\Leftrightarrow N\vDash\psi(a)$.
\end{cor}

\begin{proof}
This can be shown similarly to Lemma~\ref{efgame}.
\end{proof}

Next, through the three Lemmas, we show that if $\varphi(x)\in\mathrm{Fml}_{\mathcal{FO}}$ is bisimulation invariant, then $\varphi(x)$ is $k$-restricted $l$-bisimulation invariant.

\begin{lem}[$\mathrm{RCA}_0$~\cite{Otto2004b,MR2258711}]\label{bisimimlloc}
Let $\varphi(x)\in\mathrm{Fml}_{\mathcal{FO}}$ which has no free number variables except $x$, $\mathrm{Qr}(\varphi)=q$. Let $l=2^{q}-1$. If $\varphi(x)$ is bisimulation invariant, then $\varphi(x)$ is $l$-local. 
\end{lem}

\begin{proof}
To show that $\varphi(x)$ is $l$-local, take $M=(W, R, V)$ and $a\in W$. Then we may assume that $M$ is a tree model with root $a$ by considering the tree unravelling. Next, for all $i<q$, we define $M_{A_{i}}=(A_{i}, R^{A_{i}}, V^{A_{i}})$ and $M_{B_{i}}=(B_{i}, R^{B_{i}}, V^{B_{i}})$ where
\begin{equation*}
\begin{split}
A_{i}\coloneq \{\langle{w, i, 0}\rangle\in W\times q\times2\},\, R^{A_{i}}\coloneq\{(\langle{w, i, 0}\rangle, \langle{w^{\prime}, i, 0}\rangle)\mid wRw^{\prime}\},
\end{split}
\end{equation*}
\begin{equation*}
\begin{split}
V^{A_{i}}(\varphi(\langle{w, i, 0}\rangle))\coloneq V(\varphi(w)) \text{ for all $\varphi(x)\in\mathrm{Fml}_{\mathcal{FO}}$ and $\langle{w, i, 0}\rangle\in A_{i}$},
\end{split}
\end{equation*}
\begin{equation*}
\begin{split}
B_{i}\coloneq \{\langle{w, i, 1}\rangle\in W^{l}[a]\times q\times2\},\, R^{B_{i}}\coloneq\{(\langle{w, i, 1}\rangle, \langle{w^{\prime}, i, 1}\rangle)\mid wR^{l}[a]w^{\prime}\},
\end{split}
\end{equation*}
\begin{equation*}
\begin{split}
V^{B_{i}}(\varphi(\langle{w, i, 1}\rangle))\coloneq V^{l}[a](\varphi(w)) \text{ for all $\varphi(x)\in\mathrm{Fml}_{\mathcal{FO}}$ and $\langle{w, i, 1}\rangle\in B_{i}$}.
\end{split}
\end{equation*}

Then we define two disjoint unions of models $N=\bigsqcup_{i<q}M_{A_{i}}\sqcup \bigsqcup_{i<q}M_{B_{i}}\sqcup M$ and $N^{\prime}=\bigsqcup_{i<q}M_{A_{i}}\sqcup \bigsqcup_{i<q}M_{B_{i}}\sqcup M^{l}[a]$ (Note that the evaluation of a formula in a disjoint union of models is determined by the evaluations under each valuation of formulas effectively constructed from the original one. See \cite{MR2060932}).The following figures represent the pointed models $(N, a)$ and $(N^{\prime}, a)$, respectively.
%by the following:% for every $b\in E\cup W\cup W^{l}[a]$ and $\varphi\in\mathrm{Fml}$,
%\begin{equation*}
%C=E\sqcup W,\, R^{C}=Q\sqcup R,
%\end{equation*}
%\begin{equation*}
%U=V^{A_{i}}\sqcup V^{B_{i}}\sqcup V,
%\end{equation*}
%\begin{equation*}
%U(b, \varphi)=1\Leftrightarrow(b\in A_{i}\rightarrow V^{A_{i}}(b, \varphi)=1)\land(b\in B_{i}\rightarrow V^{B_{i}}(b, \varphi)=1)\land (b\in W\rightarrow V(b, \varphi)=1),
%\end{equation*}
%\begin{equation*}
%C^{\prime}=E\sqcup W^{l}[a],\, R^{C^{\prime}}=Q\sqcup R^{l}[a],
%\end{equation*}
%\begin{equation*}
%U^{\prime}=V^{A_{i}}\sqcup V^{B_{i}}\sqcup V^{l}[a].
%\end{equation*}
%\begin{equation*}
%U^{\prime}(b, \varphi)=1\Leftrightarrow(b\in A_{i}\rightarrow V^{A_{i}}(b, \varphi)=1)\land(b\in B_{i}\rightarrow V^{B_{i}}(b, \varphi)=1)\land (b\in W^{l}[a]\rightarrow V^{l}[a](b, \varphi)=1).
%\end{equation*}

\begin{center}
\begin{tikzpicture}[
    scale=0.85, 
    transform shape,
    >=Stealth,
    thick,
    % 文字のスタイル
    math_label/.style={font=\Large},
    q_copies/.style={font=\large, text height=1.5ex, text depth=.25ex}
]

    % ================= 左側のグループ: M (一定の間隔 2.5 で配置) =================
    % 1つ目の M (無限の錐)
    \coordinate (M1_apex) at (0,0);
    \fill[gray!10] (M1_apex) -- (-1,2) -- (1,2) -- cycle;
    \draw[thick] (-1,2) -- (M1_apex) -- (1,2);
    \node[math_label] at (0,1.2) {$M$};
    
    % 2つ目の M
    \coordinate (M2_apex) at (2.5,0);
    \fill[gray!10] (M2_apex) -- (1.5,2) -- (3.5,2) -- cycle;
    \draw[thick] (1.5,2) -- (M2_apex) -- (3.5,2);
    \node[math_label] at (2.5,1.2) {$M$};
    
    % 3つ目の M (頂点 a があるもの)
    \coordinate (M3_apex) at (5.0,0);
    \fill[gray!10] (M3_apex) -- (4,2) -- (6,2) -- cycle;
    \draw[thick] (4,2) -- (M3_apex) -- (6,2);
    \node[math_label] at (5.0,1.2) {$M$};
    % 頂点 a と点
    \node[circle, fill, inner sep=1.2pt, label=below:{\large $a$}] at (M3_apex) {};

    % M の下部の波括弧と数
    \draw[decorate, decoration={brace, amplitude=10pt, mirror}, thick] 
        (-0.2,-0.4) -- (2.7,-0.4) 
        node[q_copies, midway, below=12pt] {$q$ copies};

    % ================= 右側のグループ: M^l (3番目の錐に近づけて等間隔 2.5 で配置) =================
    % 1つ目の M^l (上が閉じた有限の錐 - 8.0 から 7.5 に近づけました)
    \coordinate (Ml1_apex) at (7.5,0);
    \fill[gray!10] (Ml1_apex) -- (6.7,1.6) -- (8.3,1.6) -- cycle;
    \draw[thick] (Ml1_apex) -- (6.7,1.6) -- (8.3,1.6) -- cycle;
    \node[math_label] at (7.5,1.0) {$M^l$};
    
    % 2つ目の M^l (10.5 から 10.0 に近づけました)
    \coordinate (Ml2_apex) at (10.0,0);
    \fill[gray!10] (Ml2_apex) -- (9.2,1.6) -- (10.8,1.6) -- cycle;
    \draw[thick] (Ml2_apex) -- (9.2,1.6) -- (10.8,1.6) -- cycle;
    \node[math_label] at (10.0,1.0) {$M^l$};

    % M^l の下部の波括弧と数 (位置の変更に合わせて波括弧の範囲も 7.3 から 10.2 に調整)
    \draw[decorate, decoration={brace, amplitude=10pt, mirror}, thick] 
        (7.3,-0.4) -- (10.2,-0.4) 
        node[q_copies, midway, below=12pt] {$q$ copies};

    % ================= 注釈の追加 =================
    \node[below=1.5cm of M3_apex, font=\small\itshape] {Note: This diagram illustrates pointed model $(N, a)$.};

\end{tikzpicture}
\end{center}

\begin{center}
\begin{tikzpicture}[
    scale=0.85, 
    transform shape,
    >=Stealth,
    thick,
    % 文字のスタイル
    math_label/.style={font=\Large},
    q_copies/.style={font=\large, text height=1.5ex, text depth=.25ex}
]

    % ================= 左側のグループ =================
    % 1つ目の M (無限の錐)
    \coordinate (M1_apex) at (0,0);
    \fill[gray!10] (M1_apex) -- (-1,2) -- (1,2) -- cycle;
    \draw[thick] (-1,2) -- (M1_apex) -- (1,2);
    \node[math_label] at (0,1.2) {$M$};
    
    % 2つ目の M (もとの無限錐 M)
    \coordinate (M2_apex) at (2.5,0);
    \fill[gray!10] (M2_apex) -- (1.5,2) -- (3.5,2) -- cycle;
    \draw[thick] (1.5,2) -- (M2_apex) -- (3.5,2);
    \node[math_label] at (2.5,1.2) {$M$};
    
    % 3つ目の錐 (左から3番目: 小さい有限錐 M^l)
    \coordinate (M3_apex) at (5.0,0);
    \fill[gray!10] (M3_apex) -- (4.2,1.6) -- (5.8,1.6) -- cycle;
    \draw[thick] (M3_apex) -- (4.2,1.6) -- (5.8,1.6) -- cycle;
    \node[math_label] at (5.0,1.0) {$M^l$};
    % 頂点 a と点
    \node[circle, fill, inner sep=1.2pt, label=below:{\large $a$}] at (M3_apex) {};

    % M1 と M2 の下部の波括弧と数 (q copies)
    \draw[decorate, decoration={brace, amplitude=10pt, mirror}, thick] 
        (-0.2,-0.4) -- (2.7,-0.4) 
        node[q_copies, midway, below=12pt] {$q$ copies};

    % ================= 右側のグループ: M^l (3番目の錐に近づけて等間隔 2.5 で配置) =================
    % 4つ目の錐 / 1つ目の M^l (8.0 から 7.5 に左シフトして接近)
    \coordinate (Ml1_apex) at (7.5,0);
    \fill[gray!10] (Ml1_apex) -- (6.7,1.6) -- (8.3,1.6) -- cycle;
    \draw[thick] (Ml1_apex) -- (6.7,1.6) -- (8.3,1.6) -- cycle;
    \node[math_label] at (7.5,1.0) {$M^l$};
    
    % 5つ目の錐 / 2つ目の M^l (10.5 から 10.0 に左シフトして接近)
    \coordinate (Ml2_apex) at (10.0,0);
    \fill[gray!10] (Ml2_apex) -- (9.2,1.6) -- (10.8,1.6) -- cycle;
    \draw[thick] (Ml2_apex) -- (9.2,1.6) -- (10.8,1.6) -- cycle;
    \node[math_label] at (10.0,1.0) {$M^l$};

    % M^l の下部の波括弧と数 (接近した座標に合わせて範囲を 7.3 から 10.2 に調整)
    \draw[decorate, decoration={brace, amplitude=10pt, mirror}, thick] 
        (7.3,-0.4) -- (10.2,-0.4) 
        node[q_copies, midway, below=12pt] {$q$ copies};

    % ================= 注釈の追加 =================
    \node[below=1.5cm of M3_apex, font=\small\itshape] {Note: This diagram illustrates pointed model $(N^{\prime}, a)$.};

\end{tikzpicture}
\end{center}

Clearly, $(M, a)\,\underline{\xleftrightarrow{}}\,(N, a)$ and $(N^{\prime}, a)\,\underline{\xleftrightarrow{}}\,(M^{l}[a], a)$. 
\vspace{2mm}
%Next, we show that $(N, a)\,{\equiv}^{q}(N^{\prime}, a)$, that is, for all $\psi(x)\in\mathrm{Fml}_{\mathcal{FO}}$ which has no free number variables except $x$ and $\mathrm{Qr}(\psi)\leq q$, $N\vDash\psi(a)\Leftrightarrow N^{\prime}\vDash\psi(a)$.

Next, we show that $\mathrm{II}$ has a winning strategy in the $q$-round game from $(N;N^{\prime})$, that is, we construct the desired function $f_{\mathrm{II}}:(W\cup W^{\prime})^{< 2q}\rightarrow W\cup W^{\prime}$ recursively. 

Given $\sigma=\langle{\sigma_{0},\dots,\sigma_{2m}}\rangle\in(C\cup C^{\prime})^{<2q}$ such that $E=\{\sigma_{2i}\mid i<m\}$ is isomorphic to $O=\{\sigma_{2i+1}\mid i<m\}$. If $\sigma_{2m}$ is within $2^{q-m}$-distance along $R^{C}\cup R^{C^{\prime}}$ from some $\sigma_{i}$ $(i<2m)$, then we can find an element $\tau\in C\cup C^{\prime}$ such that $E\cup\{\sigma_{2m}\}$ is isomorphic to $O\cup\{\tau\}$. Then we define $f_{\mathrm{II}}(\sigma\restriction(2m+1))\coloneq\tau$. If $\sigma_{2m}$ is at a distance greater than $2^{q-m}$ from all elements $\sigma_{0},\dots,\sigma_{2m-1}$, then we can take an element $\tau$ from $A_{j}$ or $B_{j}$ that has not yet been selected, that is, for all $i<2m$, $\sigma_{i}\not\in A_{j}\cup B_{j}$. Then we define $f_{\mathrm{II}}(\sigma\restriction(2m+1))\coloneq\tau$. Thus, by the above, we define the function $f_{\mathrm{II}}:(W\cup W^{\prime})^{< 2q}\rightarrow W\cup W^{\prime}$. Clearly, $\mathrm{II}$ has a winning strategy in the $q$-round game from $(N;N^{\prime})$.
\vspace{2mm}

By the above and Lemma~\ref{efgame}, $N\vDash\varphi(a)\Leftrightarrow N^{\prime}\vDash\varphi(a)$. Since $\varphi(x)$ is bisimulation invariant, $M\vDash\varphi(a)\Leftrightarrow N\vDash\varphi(a)$ and $N^{\prime}\vDash\varphi(a)\Leftrightarrow M^{l}[a]\vDash\varphi(a)$. Therefore $M\vDash\varphi(a)\Leftrightarrow M^{l}[a]\vDash\varphi(a)$, so $\varphi(x)$ is $l$-local.
\end{proof}

\begin{lem}[$\mathrm{RCA}_0$~\cite{Otto2004b,MR2258711}]\label{resbisiminv}
Let $\varphi(x)\in\mathrm{Fml}_{\mathcal{FO}}$ which has no free number variables except $x$, and $\mathrm{Qr}(\varphi)=q$. Let $l=2^{q}-1$. If $\varphi(x)$ is bisimulation invariant and $\varphi(x)$ is $l$-local, then $\varphi(x)$ is $l$-bisimulation invariant.
\end{lem}

\begin{proof}
Suppose $(M, a)\,\underline{\xleftrightarrow{}}^{l}\,(N, b)$ and $M\vDash\varphi(a)$. We show that $N\vDash\varphi(b)$. We may assume that $M$ and $N$ are $l$-locally tree models by considering the tree unravelling. By $l$-locality, $M\vDash\varphi(a)\Leftrightarrow M^{l}[a]\vDash\varphi(a)$. Since $(M^{l}[a], a)\,\underline{\xleftrightarrow{}}^{l}\,({N}^{l}[b], b)$ holds, $(M^{l}[a], a)\,\underline{\xleftrightarrow{}}\,({N}^{l}[b], b)$ also holds. Therefore ${N}^{l}[b]\vDash\varphi(b)$. By $l$-locality, $N\vDash\varphi(b)$. Therefore, $\varphi(x)$ is $l$-bisimulation invariant.
\end{proof}

\begin{lem}[$\mathrm{RCA}_0$]\label{krlbisim1}
Let $\varphi(x)\in\mathrm{Fml}_{\mathcal{FO}}$ which has no free number variables except $x$, contains only $P_{0}(y), \dots, P_{k-1}(y)$, and $r(y, y^{\prime})$ as first-order predicates, and $\mathrm{Qr}(\varphi)=q$. Let $l=2^{q}-1$. If $\varphi(x)$ is $l$-bisimulation invariant, then $\varphi(x)$ is $k$-restricted $l$-bisimulation invariant.
\end{lem}

\begin{proof}
Suppose $(M, a)\,\underline{\xleftrightarrow{}}^{k, l}_{*}\,(N, b)$ and $M\vDash\varphi(a)$. We may assume that $(M, a)$ and $(N, b)$ are tree models with roots $a$ and $b$, respectively, by considering the tree unravelling. Let $M=(W, R, V)$. Then we define the new valuation $V^{\star}$ and define $M^{\star}=(W, R, V^{\star})$ such that $M^{\star}\vDash\varphi(a)$.

Let $\mathcal{P}=\{P_{0}(y), \dots, P_{k-1}(y)\}$. First, we define the function $(\cdot)_{\top}:\mathrm{Fml}_{\mathcal{FO}}\rightarrow\mathrm{Fml}_{\mathcal{FO}}$ by the following:
\begin{itemize}
    \item $(P_{i}(x))_{\top}=P_{i}(x)$ if $P_{i}(x)\in\mathcal{P}$,
    \item $(P_{i}(x))_{\top}=\top$ if $P_{i}(x)\not\in\mathcal{P}$,
    \item $(\bot)_{\top}=\bot$,
    \item $(\psi\rightarrow\theta)_{\top}=(\psi)_{\top}\rightarrow(\theta)_{\top}$,
    \item $(\forall x\psi)_{\top}=\forall x(\psi)_{\top}$.
\end{itemize}
We define such $(\cdot)_{\top}$ using $\mathrm{I}\mathrm{\Sigma}^0_1$. Then we define $V^{\star}:T_{M}\cup S_{M}\rightarrow M\cup2$ by the following: for all $w\in W$ and $\psi\in\mathrm{Fml}_{\mathcal{FO}}$, $V^{\star}(\psi(w))=V((\psi(w))_{\top})$. Then, for all $w\in W$ and $\psi\in\mathrm{Fml}$, 
\begin{equation*}
\begin{split}
V^{\star}(\forall x\psi(x, w))=1&\iff V((\forall x\psi(x, w))_{\top})=1\\&\iff V(\forall x(\psi(x, w))_{\top})=1\\&\iff \forall u\in W(V((\psi(u, w))_{\top})=1)\\&\iff \forall u\in W(V^{\star}(\psi(u, w))=1).
\end{split}
\end{equation*}

Considering the identity map from $(M, a)$ to $(M^{\star}, a)$, $\mathrm{II}$ has a winning strategy in the $q$-round $k$-restricted Ehrenfeucht-Fra\"{i}ss\'{e} game from $(M;M^{\star})$. By Lemma~\ref{k-efgame}, $M^{\star}\vDash\varphi(a)$. Similarly, for $(N, b)$, we define $(N^{\star}, b)$. By $(M, a)\,\underline{\xleftrightarrow{}}^{k, l}_{*}\,(N, b)$, $(M^{\star}, a)\,\underline{\xleftrightarrow{}}^{k, l}_{*}\,(N^{\star}, b)$, and by the construction, $(M^{\star}, a)\,\underline{\xleftrightarrow{}}^{l}\,(N^{\star}, b)$. Since $\varphi(x)$ is $l$-bisimulation invariant, $N^{\star}\vDash\varphi(b)$. Considering the identity map from $(N, b)$ to $(N^{\star}, b)$, $\mathrm{II}$ has a winning strategy in the $q$-round $k$-restricted Ehrenfeucht-Fra\"{i}ss\'{e} game from $(N;N^{\star})$. By Lemma~\ref{k-efgame}, $N\vDash\varphi(b)$. 
\end{proof}

\begin{thm}[The van Benthem characterization theorem (the semantic form), $\mathrm{RCA}_0$]\label{vanB}
Let $\varphi(x)\in\mathrm{Fml}_{\mathcal{FO}}$ which has no free number variables except $x$. Then the following are equivalent:
\begin{enumerate}
    \item $\varphi(x)$ is bisimulation invariant,
    \item there exists $\psi\in\mathrm{Fml}$ such that $\vDash\forall x(\varphi(x)\leftrightarrow\mathrm{ST}_{x}(\psi))$.
\end{enumerate}
\end{thm}

\begin{proof}
By Proposition~\ref{2im1vB1}, it is enough to show that for every $\varphi(x)\in\mathrm{Fml}_{\mathcal{FO}}$ which has no free number variables except $x$, if $\varphi(x)$ is bisimulation invariant, then there exists $\psi\in\mathrm{Fml}$ such that $\mathrm{Pbl}_{\mathcal{FO}}\bigl(\emptyset, \forall x(\varphi(x)\leftrightarrow\mathrm{ST}_{x}(\psi))\bigr)$ holds.

Let $\varphi(x)\in\mathrm{Fml}_{\mathcal{FO}}$ which has no free number variables except $x$ and $\mathrm{Qr}(\varphi)=q$. Let $l=2^{q}-1$. We assume that $\varphi(x)$ contains only $P_{0}(y), \dots, P_{k-1}(y)$, and $r(y, y^{\prime})$ as first-order predicates. 
\vspace{2mm}

We make the desired modal formula $\psi$. Let $P=\{p_{0}, \dots, p_{k-1}\}$ and let $\mathrm{Fml}[P]$ be the set of all modal formulas containing only $p_{0},\dots,p_{k-1}$ as the propositions.
\vspace{1mm}

Let $K^{P}$ be the set of all finite tree Kripke models with a root. We denote $K^{P}=\{(K_{i}, \rightarrow_{i}, v_{i}, a_{i})\}_{i}$, where for each $i$, $K_{i}$ is the set of worlds, $\rightarrow_{i}$ is the accessibility relation, $v_{i}$ is the valuation from $K_{i}\times P$ to $2$, and $a_{i}$ is the root of the model $(K_{i}, \rightarrow_{i}, V_{i})$, where $V_{i}$ is the valuation from $K_{i}\times\mathrm{Fml}[P]$ to $2$ extending $v_{i}$. Let $K^{P}_{1}=\{(K, \rightarrow, v, a)\in K^{P}\mid |K|=1\}$. Then we may consider that $K^{P}_{1}$ is finite since, for a point $a$, the number of the valuations from $\{a\}\times P$ to $2$ is at most $2^{k}$. 
\vspace{1mm}

Fix $S=\{(K_{i}, \rightarrow_{i}, v_{i}, a_{i})\}_{i\in I}\subseteq_{fin} K^{P}$ (that is, $I$ is finite) and $a=(\{a\}, \emptyset, v_{a}, a)\in K^{P}_{1}$. Then we define the finite tree model with the root $a$, $aS=(K, \rightarrow_{K}, v_{K}, a)$, by the following:
\begin{equation*}
K=\{a\}\sqcup\bigsqcup_{i\in I}K_{i},
\end{equation*}
\begin{equation*}
\rightarrow_{K}=\{(a, a_{i})\mid i\in I\}\sqcup\bigsqcup_{i\in I}\rightarrow_{i},
\end{equation*}
\begin{equation*}
v_{K}=v_{a}\cup\bigcup_{i\in I}v_{i}.
\end{equation*}

For each $i\leq l$, we define $K^{P}_{i+1}$ as follows: for defined $K^{P}_{i}$, $K^{P}_{i+1}=\{aS\mid a\in K^{P}_{1}\land S\subseteq_{fin}K^{P}_{i}\}$. For each $i\leq l$, $|K^{P}_{i+1}|\leq 2^{|K^{P}_{i}|}\times|K^{P}_{1}|$.
\vspace{1mm}

Next, $a\in K^{P}_{1}$, and $S\subseteq_{fin}K^{P}$, we define $\chi_{aS}$ as follows: for each $a\in K^{P}_{1}$, 
\begin{equation*}
\chi_{a}\equiv\bigwedge\{p\mid p\in P\land v(a, p)=1\}\land\bigwedge\{\neg p\mid p\in P\land v(a, p)=0\}.
\end{equation*}

If $\chi_{K}$ is defined for each $K\in S\subseteq_{fin}K^{P}$, then
\begin{equation*}
\chi_{aS}\equiv\chi_{a}\land(\bigwedge_{K\in S}\Diamond\chi_{K})\land\Box(\bigvee_{K\in S}\chi_{K}).
\end{equation*} 

Then let $\psi\equiv\bigvee\{\chi_{aS}\mid aS\in K^{P}_{l+1}, aS\vDash\varphi(a)\}$. We show that $\psi$ is the desired formula. First, we show the following lemma.

\begin{lem}[$\mathrm{RCA}_0$]\label{mmm}
Let $M=(W, R, V)$ be a tree model with the root $w\in W$, whose height is $l$. Let $aS=(K, \rightarrow_{K}, v_{K}, a)\in K^{P}_{l+1}$ such that the height of $aS$ is $l$. Then the following are equivalent:
\begin{enumerate}
    \item $(aS, a)\,\underline{\xleftrightarrow{}}^{k, l}_{*}\,(M, w)$,
    \item $M, w\Vdash\chi_{aS}$.
\end{enumerate} 
\end{lem}
\begin{lemproof}
$(1)\Rightarrow(2)$: For any model $aS\in K^{P}$, let $ht(aS)$ be the height of $aS$. Let $Z_{0}\subseteq\dots\subseteq Z_{l}\subseteq K\times W$ be a $k$-restricted $l$-bisimulation over $(aS, a)$ and $(M, w)$. By the construction of $K^{P}_{l}$, for each $b\in K$, there exists $S_{b}\subseteq_{fin}K^{P}$ such that $bS_{b}=(K_{b}, \rightarrow_{K_{b}}, v_{K_{b}}, b)$ is a submodel of $aS$. Then we show the following statement by induction on $m\leq l$:
\begin{equation*}
\forall b\in K\forall s\in W(bZ_{l-m} s\land(ht(bS_{b})\leq m)\rightarrow M, s\Vdash\chi_{bS_{b}}).
\end{equation*}

We may show it by $\mathrm{I}\Sigma^{0}_{1}$. 
\begin{itemize}
    \item \textit{Base case}: Fix $b\in K$ and $s\in W$ with $bZ_{l}s\land ht(bS_{b})=0$. Then $bS_{b}=b=(\{b\}, \emptyset, v_{b}, b)$. Then, for all $p\in P$, $v_{b}(b, p)=1\leftrightarrow V(s, p)=1$. Thus $M, s\Vdash\chi_{bS_{b}}$.
    \item \textit{Induction step}: We assume that $\forall b\in K\forall s\in W(bZ_{l-m} s\land(ht(bS_{b})\leq m)\rightarrow M, s\Vdash\chi_{bS_{b}})$. Fix $c\in K$ and $u\in W$ with $cZ_{l-(m+1)}u\land ht(cS_{c})\leq m+1$. We may assume that $ht(cS_{c})=m+1$. We show that $M, u\Vdash\chi_{cS_{c}}$. Clearly, $M, u\Vdash\chi_{c}$. Let $cS_{c}=(K_{c}, \rightarrow_{K_{c}}, v_{K_{c}}, c)$. 
\begin{itemize}
    \item Take $K^{\prime}\in S_{c}$. Then there exists $c^{\prime}\in c\rightarrow_{K_{c}}$ and $S_{c^{\prime}}\subseteq_{fin}K^{P}$ such that $K^{\prime}=c^{\prime}S_{c^{\prime}}$. Note that $c^{\prime}S_{c^{\prime}}$ is the submodel of $aS$. By $k$-restricted $l$-bisimulation, there exists $u^{\prime}\in uR$ such that $c^{\prime}Z_{l-m}u^{\prime}$. Then, by the assumption, $M, u^{\prime}\Vdash\chi_{c^{\prime}S_{c^{\prime}}}$, so $M, u^{\prime}\Vdash\chi_{K^{\prime}}$. Thus $M, u\Vdash\Diamond\chi_{K^{\prime}}$. Therefore, $M, u\Vdash\bigwedge_{K^{\prime}\in S_{c}}\Diamond\chi_{K^{\prime}}$. 
    \item Take $u^{\prime}\in uR$. By $k$-restricted $l$-bisimulation, there exists $c^{\prime}\in c\rightarrow_{K_{c}}$ such that $c^{\prime}Z_{l-m}u^{\prime}$. Then there exists $S_{c^{\prime}}\subseteq_{fin}K^{P}$ such that $c^{\prime}S_{c^{\prime}}\in S_{c}$. By the assumption, $M, u^{\prime}\Vdash\chi_{c^{\prime}S_{c^{\prime}}}$. Thus $M, u^{\prime}\Vdash\bigvee_{K\in S_{c}}\chi_{K}$. Therefore, $M, u\Vdash\Box(\bigvee_{K\in S_{c}}\chi_{K})$. 
\end{itemize}
By the above, $M, u\Vdash\chi_{cS_{c}}$.
\end{itemize}
By induction, for all $m\leq l$, $\forall b\in K\forall s\in W(bZ_{l-m} s\land(ht(bS_{b})\leq m)\rightarrow M, s\Vdash\chi_{bS_{b}})$. In particular, $M, w\Vdash\chi_{aS}$ (Note that $aS=aS_{a}$).
\vspace{3mm}

$(2)\Rightarrow(1)$: For each $m\leq l$, $Z_{m}=\{(b, u)\in K\times W\mid M, u\Vdash\chi_{bS_{b}}\land ht(bS_{b})\geq l-m\}$. We show that $Z_{0}\subseteq\dots\subseteq Z_{l}$ is a $k$-restricted $l$-bisimulation over $(aS, a)$ and $(M, w)$. 
\begin{itemize}
    \item \textit{Base condition $1$}: Note that $aS=aS_{a}$. Since ($2$) holds, $aZ_{0}w$. 
    \item \textit{Base condition $2$}: For all $p\in P$ and $(b, s)\in Z_{l}$, $v_{K_{b}}(b, p)=1\leftrightarrow V(s, p)=1$ since $M, s\Vdash\chi_{b}$ holds. 
    \item \textit{Forth condition}: Fix $i<l$ and $(b, u)\in Z_{i}$ with $bS_{b}=(K_{b}, \rightarrow_{K_{b}}, v_{K_{b}}, b)$ and $ht(bS_{b})\geq l-i>0$, and fix $c\in b\rightarrow_{K_{b}}$. Then there exists $S_{c}\subseteq_{fin}K^{P}$ such that $cS_{c}\in S_{b}$. Since $M, u\Vdash\chi_{bS_{b}}$, $M, u\Vdash\bigwedge_{K\in S_{b}}\Diamond\chi_{K}$. Therefore, $M, u\Vdash\Diamond\chi_{cS_{c}}$, so there exists $u^{\prime}\in uR$ such that $M, u^{\prime}\Vdash\chi_{cS_{c}}$. Thus $cZ_{i+1}u^{\prime}$. 
    \item \textit{Back condition}: Fix $i<l$ and $(b, u)\in Z_{i}$ with $bS_{b}=(K_{b}, \rightarrow_{K_{b}}, v_{K_{b}}, b)$ and $ht(bS_{b})\geq l-i>0$, and fix $u^{\prime}\in uR$. Since $M, u\Vdash\chi_{bS_{b}}$, $M, u\Vdash\Box\bigvee_{K\in S_{b}}\chi_{K}$, so $M, u^{\prime}\Vdash\bigvee_{K\in S_{b}}\chi_{K}$. Thus there exists $K\in S_{b}$ such that $M, u^{\prime}\Vdash\chi_{K}$. Then there exists $c\in b\rightarrow_{K_{b}}$ and $S_{c}\subseteq_{fin}K^{P}$ such that $K=cS_{c}$. Since $M, u^{\prime}\Vdash\chi_{cS_{c}}$, $cZ_{i+1}u^{\prime}$. 
\end{itemize}
By the above, $Z_{0}\subseteq\dots\subseteq Z_{l}$ is an $l$-bisimulation over $(aS, a)$ and $(M, w)$.
\end{lemproof}

Last, we show that $\vDash\forall x(\varphi(x)\leftrightarrow\mathrm{ST}_{x}(\psi))$ holds. Fix a model $M=(W, R, V)$ and $w\in W$. We may assume that $M$ is a tree model with root $w$, whose height is $l$ (if the height is less than $l$, we can extend it). For each $u\in W$, $ht_{w}(u)$ is the length of the path from $w$. We show that $M\vDash\varphi(w)$ if and only if $M, w\Vdash\psi$. 
\vspace{2mm}

First, we assume that $M, w\Vdash\psi$. Then there exists $aS\in K^{P}_{l+1}$ such that $aS\vDash\varphi(a)$ and $M, w\Vdash\chi_{aS}$. By Lemma~\ref{mmm}, $(aS, a)\,\underline{\xleftrightarrow{}}^{k, l}_{*}\,(M, w)$. By Lemma~\ref{bisimimlloc} and Lemma~\ref{resbisiminv}, $\varphi(x)$ is $k$-restricted $l$-bisimulation invariant. Thus, $M\vDash\varphi(w)$.
\vspace{2mm}

Next, we assume that $M\vDash\varphi(w)$. To derive a contradiction, we assume that for all $aS\in K^{P}_{l+1}$, $M, w\not\Vdash\chi_{aS}$. We show the following lemma:
\begin{lem}[$\mathrm{RCA}_0$]
For all $m\leq l$, there exists $u\in W$ with $ht_{w}(u)=m$ such that for all $aS\in K^{P}_{l+1-m}$, $M, u\not\Vdash\chi_{aS}$. 
\end{lem}
\begin{lemproof}
We prove it by induction on $m\leq l$. 
\vspace{2mm}

\textit{Base case}: For $m=0$, by the assumption, for all $aS\in K^{P}_{l+1}$, $M, w\not\Vdash\chi_{aS}$. 
\vspace{2mm}

\textit{Induction step}: We assume that there exists $u\in W$ with $ht_{w}(u)=m$ such that for all $aS\in K^{P}_{l+1-m}$, $M, u\not\Vdash\chi_{aS}$. If there exists no $u^{\prime}\in uR$, then there exists $a\in K^{P}_{l+1-m}$ ($a\in K^{P}_{1}$) such that $M, u\Vdash\chi_{a}$. It is a contradiction. Thus, $uR$ is not empty. To deduce the contradiction, we assume that for all $u^{\prime}\in uR$, there exists $aS\in K^{P}_{l-m}$ such that $M, u^{\prime}\Vdash\chi_{aS}$. Then $S^{\prime}=\{aS\in K^{P}_{l-m}\mid M, u\Vdash\Diamond\chi_{aS}\}$ is not empty. Now, we consider $u$ as the element of $K^{P}_{1}$, that is, $u=(\{u\}, \emptyset, v_{u}, u)\in K^{P}_{1}$. Then $uS^{\prime}\in K^{P}_{l+1-m}$ and $M, u\Vdash\chi_{uS^{\prime}}$. It is a contradiction. Thus there exists $u^{\prime}\in uR$ such that for all $aS\in K^{P}_{l-m}$, $M, u^{\prime}\not\Vdash\chi_{aS}$. By induction, for all $m\leq l$, there exists $u\in W$ with $ht_{w}(u)=m$ such that for all $aS\in K^{P}_{l+1-m}$, $M, u\not\Vdash\chi_{aS}$. 
\end{lemproof}

In particular, there exists $u\in W$ with $ht_{w}(u)=l$ such that for all $a\in K^{P}_{1}$, $M, u\not\Vdash\chi_{a}$. It is a contradiction. Thus there exists $aS\in K^{P}_{l+1}$ such that $M, w\Vdash\chi_{aS}$. By Lemma~\ref{mmm}, $(aS, a)\,\underline{\xleftrightarrow{}}^{k, l}_{*}\,(M, w)$. Since $\varphi(x)$ is $k$-restricted $l$-bisimulation invariant, $aS\vDash\varphi(a)$. Therefore, $M, w\Vdash\psi$.
%By G\"odel's Completeness Theorem, that is, by $\mathrm{WKL}_{0}$, $\mathrm{Pbl}_{\mathcal{FO}}\bigl(\emptyset, \forall x(\varphi(x)\leftrightarrow\mathrm{ST}_{x}(\psi))\bigr)$ holds.
\end{proof}

%\begin{comment}
According to the above, we obtain the following result.

\begin{cor}[The van Benthem characterization theorem (the hybrid form), $\mathrm{WKL}_0$]\label{vanB}
Let $\varphi(x)\in\mathrm{Fml}_{\mathcal{FO}}$ which has no free number variables except $x$. Then the following are equivalent:
\begin{enumerate}
    \item $\varphi(x)$ is bisimulation invariant,
    \item there exists $\psi\in\mathrm{Fml}$ such that $\mathrm{Pbl}_{\mathcal{FO}}\bigl(\emptyset, \forall x(\varphi(x)\leftrightarrow\mathrm{ST}_{x}(\psi))\bigr)$ holds.
\end{enumerate}
\end{cor}
%\end{comment}

%Note that we only use $\mathrm{WKL}_0$ to show that $\mathrm{Pbl}_{\mathcal{FO}}\bigl(\emptyset, \forall x(\varphi(x)\Leftrightarrow\mathrm{ST}_{x}(\psi))\bigr)\leftrightarrow\,\vDash\forall x(\varphi(x)\leftrightarrow\mathrm{ST}_{x}(\psi))$.
The following theorem is known as Rosen's theorem~\cite{MR1468908}.

\begin{cor}[Rosen's theorem, $\mathrm{RCA}_0$~\cite{MR1468908}]
Let $\varphi(x)\in\mathrm{Fml}_{\mathcal{FO}}$ which has no free number variables except $x$. Then the following are equivalent:
\begin{enumerate}
    \item $\varphi(x)$ is bisimulation invariant for finite models,
    \item there exists $\psi\in\mathrm{Fml}$ such that for every finite model $M$ on $\mathcal{L}_{\mathcal{FO}}$, $M\vDash\forall x(\varphi(x)\leftrightarrow\mathrm{ST}_{x}(\psi))$ holds.
\end{enumerate}
\end{cor}

\subsection{Syntactic form of the van Benthem characterization theorem}
Now, we show the following proposition.

\begin{prop}[$\mathrm{WKL}_0$]\label{pp}
Let $\varphi(x)\in\mathrm{Fml}_{\mathcal{FO}}$ which has no free number variables except $x$. Then the statement ``$\varphi(x)$ is bisimulation invariant'' can be expressed by a $\Sigma^0_1$ sentence.
\end{prop}

\begin{proof}
Let ${\mathcal{L}}^{\prime}_{\mathcal{FO}}$ be the first-order language (with equality) which has unary predicates $P^{\prime}_0, P^{\prime}_1, \dots$ corresponding to the proposition letters $p_0, p_1, \dots$, and a binary relation symbol $r^{\prime}$. Let $\hat{\mathcal{L}}=\mathcal{L}_{\mathcal{FO}}\cup{\mathcal{L}}^{\prime}_{\mathcal{FO}}\cup\{D_{\mathcal{L}_{\mathcal{FO}}}, D_{{\mathcal{L}}^{\prime}_{\mathcal{FO}}}\}\cup\{Z\}$ where $D_{\mathcal{L}_{\mathcal{FO}}}$ and $D_{{\mathcal{L}}^{\prime}_{\mathcal{FO}}}$ are unary predicates and $Z$ is a $2$-ary relation symbol corresponding to the bisimulation. In a similar way as $\mathcal{L}_{\mathcal{FO}}$, we may show in $\mathrm{RCA}_0$ that there exist sets $\mathrm{Trm}_{\mathcal{FO}(\hat{\mathcal{L}})}$, $\mathrm{Fml}_{\mathcal{FO}(\hat{\mathcal{L}})}$, $\mathrm{Snt}_{\mathcal{FO}(\hat{\mathcal{L}})}$, and $\mathrm{Axm}_{\mathcal{FO}(\hat{\mathcal{L}})}$ consisting of all $\hat{\mathcal{L}}$-terms, $\hat{\mathcal{L}}$-formulas, $\hat{\mathcal{L}}$-sentences, and axioms of first-order logic in $\hat{\mathcal{L}}$ respectively. Moreover, as in the case of modal logic, we can define the provability predicate of first-order logic $\mathrm{Pbl}_{\mathcal{FO}(\hat{\mathcal{L}})}(X, \psi)$ in $\mathrm{RCA}_0$. 
\vspace{2mm}

Let $T_{Z}$ be the set of the following $\hat{\mathcal{L}}$-sentences:
\begin{itemize}
    \item $\exists y\exists z Z(y, z)$,
    \item $\forall y\forall z(Z(y, z)\rightarrow D_{\mathcal{L}_{\mathcal{FO}}}(y)\land D_{{\mathcal{L}}^{\prime}_{\mathcal{FO}}}(z))$,
    \item $\forall y\forall z(Z(y, z)\rightarrow(P_{i}(y)\leftrightarrow P^{\prime}_{i}(z)))$ for each $i$,
    \item $\forall y\forall z\forall y^{\prime}(D_{\mathcal{L}_{\mathcal{FO}}}(y^{\prime})\land Z(y, z)\land r(y, y^{\prime})\rightarrow\exists z^{\prime}(D_{{\mathcal{L}}^{\prime}_{\mathcal{FO}}}(z^{\prime})\land Z(y^{\prime}, z^{\prime})\land r^{\prime}(z, z^{\prime})))$,
    \item $\forall y\forall z\forall z^{\prime}(D_{{\mathcal{L}}^{\prime}_{\mathcal{FO}}}(z^{\prime})\land Z(y, z)\land r^{\prime}(z, z^{\prime})\rightarrow\exists y^{\prime}(D_{\mathcal{L}_{\mathcal{FO}}}(y^{\prime})\land Z(y^{\prime}, z^{\prime})\land r(y, y^{\prime})))$.
\end{itemize}
$T_Z$ is the set of sentences in $\hat{\mathcal{L}}$ expressing that $Z$ is a bisimulation between $D_{\mathcal{L}}$ and $D_{\mathcal{L}'}$. Clearly, $T_{Z}$ is consistent.

Next, we define the translations $\pi,{\pi}^{\prime}:\mathcal{L}_{\mathcal{FO}}\rightarrow\mathrm{Fml}_{\mathcal{FO}(\hat{\mathcal{L}})}$ as follows:
\begin{equation*}
{P^{\pi}(z)}\equiv\pi(P(z))=P(z)\land D_{\mathcal{L}_{\mathcal{FO}}}(z),
\end{equation*}
\begin{equation*}
r^{\pi}(y, z)\equiv \pi(r(y, z))=r(y, z)\land  D_{\mathcal{L}_{\mathcal{FO}}}(y)\land D_{\mathcal{L}_{\mathcal{FO}}}(z),
\end{equation*}
\begin{equation*}
{P^{{\pi}^{\prime}}(y)}\equiv{\pi}^{\prime}(P(y))=P^{\prime}(y)\land D_{{\mathcal{L}}^{\prime}_{\mathcal{FO}}}(y),
\end{equation*}
\begin{equation*}
r^{{\pi}^{\prime}}(y, z)\equiv{\pi}^{\prime}(r(y, z))=r^{\prime}(y, z)\land D_{{\mathcal{L}}^{\prime}_{\mathcal{FO}}}(y)\land D_{{\mathcal{L}}^{\prime}_{\mathcal{FO}}}(z).
\end{equation*}

Then we may extend the translations from $\mathrm{Fml}_{\mathcal{FO}}$ to $\mathrm{Fml}_{\mathcal{FO}(\hat{\mathcal{L}})}$ as follows:
\begin{itemize}
    \item $(\varphi_{0}\land\varphi_{1})^{\pi}\equiv{\varphi_{0}}^{\pi}\land{\varphi_{1}}^{\pi}$, \,$(\varphi_{0}\land\varphi_{1})^{{\pi}^{\prime}}\equiv{\varphi_{0}}^{{\pi}^{\prime}}\land{\varphi_{1}}^{{\pi}^{\prime}}$,
    \item $(\varphi_{0}\lor\varphi_{1})^{\pi}\equiv{\varphi_{0}}^{\pi}\lor{\varphi_{1}}^{\pi}$, \,$(\varphi_{0}\lor\varphi_{1})^{{\pi}^{\prime}}\equiv{\varphi_{0}}^{{\pi}^{\prime}}\lor{\varphi_{1}}^{{\pi}^{\prime}}$,
    \item $(\neg\varphi)^{\pi}\equiv\neg\varphi^{\pi}$,\,$(\neg\varphi)^{{\pi}^{\prime}}\equiv\neg\varphi^{{\pi}^{\prime}}$,
    \item $(\varphi_{0}\rightarrow\varphi_{1})^{\pi}\equiv{\varphi_{0}}^{\pi}\rightarrow{\varphi_{1}}^{\pi}$,\,$(\varphi_{0}\rightarrow\varphi_{1})^{{\pi}^{\prime}}\equiv{\varphi_{0}}^{{\pi}^{\prime}}\rightarrow{\varphi_{1}}^{{\pi}^{\prime}}$,
    \item $(\forall z\varphi(z))^{\pi}\equiv\forall z(D_{\mathcal{L}_{\mathcal{FO}}}(z)\rightarrow(\varphi(z))^{\pi})$,\,$(\forall z\varphi(z))^{{\pi}^{\prime}}\equiv\forall z(D_{{\mathcal{L}}^{\prime}_{\mathcal{FO}}}(z)\rightarrow(\varphi(z))^{{\pi}^{\prime}})$,
    \item $(\exists z\varphi(z))^{\pi}\equiv\exists z(D_{\mathcal{L}_{\mathcal{FO}}}(z)\land(\varphi(z))^{\pi})$,\,$(\exists z\varphi(z))^{{\pi}^{\prime}}\equiv\exists z(D_{{\mathcal{L}}^{\prime}_{\mathcal{FO}}}(z)\land(\varphi(z))^{{\pi}^{\prime}})$.
\end{itemize}

Let $\tilde{\varphi}\equiv\forall y\forall z(Z(y, z)\rightarrow(\varphi^{\pi}(y)\leftrightarrow\varphi^{{\pi}^{\prime}}(z)))$. Then $\varphi(x)$ is bisimulation invariant if and only if $T_{Z}\vDash\tilde{\varphi}$, that is, for every standard model $M\vDash\tilde{\varphi}$ on $\hat{\mathcal{L}}$, if $M$ is a model of $T_{Z}$, then $M$ satisfies $\tilde{\varphi}$. Thus, by G\"odel's completeness theorem, i.e., by $\mathrm{WKL}_{0}$, $T_{Z}\vDash\tilde{\varphi}$ if and only if $\mathrm{Pbl}_{\mathcal{FO}(\hat{\mathcal{L}})}(T_{Z}, \tilde{\varphi})$. Therefore, a $\Sigma^0_1$ sentence can express the condition that $\varphi(x)$ is bisimulation invariant.
\end{proof}

Using Proposition~\ref{pp}, we introduce the van Benthem characterization theorem (the syntactic form).  

\begin{state}[The van Benthem characterization theorem (the syntactic form)]\label{vanB2}
Let $\varphi(x)\in\mathrm{Fml}_{\mathcal{FO}}$ which has no free number variables except $x$. Then the following are equivalent:
\begin{enumerate}
    \item $\mathrm{Pbl}_{\mathcal{FO}(\hat{\mathcal{L}})}(T_{Z}, \tilde{\varphi})$ holds,
    \item there exists $\psi\in\mathrm{Fml}$ such that $\mathrm{Pbl}_{\mathcal{FO}}\bigl(\emptyset, \forall x(\varphi(x)\leftrightarrow\mathrm{ST}_{x}(\psi))\bigr)$ holds.
\end{enumerate}
\end{state}

\begin{cor}[$\mathrm{WKL}_0$]\label{vanB2}
The van Benthem characterization theorem (the syntactic form) holds.
\end{cor}

\begin{proof}
By Proposition~\ref{pp} and Corollary~\ref{vanB}.
\end{proof}

\begin{thm}
The van Benthem characterization theorem (the syntactic form) is also provable in $\mathrm{PRA}$.
\end{thm}

\begin{proof}
A $\Pi^0_2$ sentence expresses the van Benthem characterization theorem (the syntactic form). By Corollary~\ref{vanB2} and the fact that $\mathrm{WKL}_0$ is a $\Pi^0_2$-conservative extension of $\mathrm{PRA}$, the van Benthem characterization theorem (the syntactic form) can also be proved in  $\mathrm{PRA}$.
\end{proof}

\subsection{Hybrid form of the van Benthem characterization theorem}
In this subsection, we show that the hybrid form of the van Benthem characterization theorem is equivalent to the lightface version of $\mathrm{WKL}$ over $\mathrm{RCA}_0$.

\begin{dfn}[$\mathrm{RCA}_0$]
Let $\varphi(\vec{x})\in\mathrm{Fml}_{\mathcal{FO}}$ and $P_{k}(y)$. Then we denote the formula $\varphi^{P_{k}}(\vec{x})$ as follows:
\begin{itemize}
    \item $(r(x, y))^{P_{k}}\equiv r(x, y)\land P_{k}(x)\land P_{k}(y)$.
    \item $(P_{i}(x))^{P_{k}}\equiv P_{i}(x)\land P_{k}(x)$ (for all $i$),
    \item $(\varphi\rightarrow \psi)^{P_{k}}\equiv\varphi^{P_{k}}\rightarrow \psi^{P_{k}}$,
    \item $(\forall y\varphi(\vec{x}, y))^{P_{k}}\equiv\forall y(P_{k}(y)\rightarrow\varphi(\vec{x}, y)^{P_{k}})$,
%    \item $(\exists y\varphi(\vec{x}, y))^{P_{k}}\equiv\exists y(P_{k}(y)\land\varphi(\vec{x}, y)^{P_{k}})$.
\end{itemize}
\end{dfn}

\begin{lem}[$\mathrm{RCA}_0$]\label{unsbisim}
Let $\psi\in\mathrm{Snt}_{\mathcal{FO}}$ which contains only $P_{0}(y), \dots, P_{k-1}(y)$, and $r(y, y^{\prime})$ as first-order predicates. Then $\psi$ is unsatisfiable on $\mathcal{L}_{\mathcal{FO}}$ if and only if $\psi^{P_{k}}\land r(x, x)\land\neg P_{k}(x)$ is bisimulation invariant.
\end{lem}

\begin{proof}
Let $\psi\in\mathrm{Snt}_{\mathcal{FO}}$ which contains only $P_{0}(y), \dots, P_{k-1}(y)$, and $r(y, y^{\prime})$ as first-order predicates. If $\psi$ is unsatisfiable, then $\psi^{P_{k}}\land r(x, x)\land\neg P_{k}(x)$ is bisimulation invariant by the definition of bisimulations. We assume that $\psi$ is satisfiable, that is, there exists a model $M=(W, R, V)$ such that $M\vDash\psi$. Similarly to Lemma~\ref{krlbisim1}, we define $M^{\star}=(W, R, V^{\star})$ (that is, $M^{\star}\vDash\forall xP_{k}(x)$). Then, by Lemma~\ref{k-efgame}, $M^{\star}\vDash\psi$, so, $M^{\star}\vDash\psi^{P_{k}}$. Next, we define two models $M_{A}=(\{a\}, R_{a}, V_{a})$ and $M_{B}=(\{b_{0}, b_{1}\}, R_{b}, V_{b})$ as follows: 
\begin{itemize}
    \item $R_{a}=\{(a, a)\}$, $V_{a}(\neg P_{k}(a))=1$, 
    \item $R_{b}=\{(b_{0}, b_{1}), (b_{1}, b_{0})\}$, $V_{b}(\neg P_{k}(b_{0}))=1$, and $V_{b}(\neg P_{k}(b_{1}))=1$.
\end{itemize}
We consider two disjoint unions of models $N=M^{\star}\sqcup M_{A}$ and $N^{\prime}=M^{\star}\sqcup M_{B}$. Then $(N, a)\,\underline{\xleftrightarrow{}}\,(N^{\prime}, b_{0})$ and $N\vDash\psi^{P_{k}}\land r(a, a)\land\neg P_{k}(a)$ and $N^{\prime}\not\vDash\psi^{P_{k}}\land r(b_{0}, b_{0})\land\neg P_{k}(b_{0})$. Thus, $\psi^{P_{k}}\land r(x, x)\land\neg P_{k}(x)$ is not bisimulation invariant.
\end{proof}

\begin{lem}[$\mathrm{WKL}_0$]\label{plrarrow}
Let $\psi\in\mathrm{Snt}_{\mathcal{FO}}$ which contains only $P_{0}(y), \dots, P_{k-1}(y)$, and $r(y, y^{\prime})$ as first-order predicates, and let $\theta(x)\equiv\psi^{P_{k}}\land r(x, x)\land\neg P_{k}(x)$. Then $\mathrm{Pbl}_{\mathcal{FO}(\hat{\mathcal{L}})}(T_{Z}, \tilde{\theta})$ holds if and only if $\mathrm{Pbl}_{\mathcal{FO}}(\psi, \bot)$ holds.
\end{lem}

\begin{proof}
By G\"odel's Completeness Theorem for $\mathcal{FO}$ (that is, $\mathrm{WKL}_0$), $\psi$ is unsatisfiable if and only if $\mathrm{Pbl}_{\mathcal{FO}}(\psi, \bot)$ holds. By Corollary~\ref{vanB} and Corollary~\ref{vanB2} (that is, $\mathrm{WKL}_0$), $\theta(x)$ is bisimulation invariant if and only if $\mathrm{Pbl}_{\mathcal{FO}(\hat{\mathcal{L}})}(T_{Z}, \tilde{\theta})$ holds. Then, by Lemma~\ref{unsbisim}, $\psi$ is unsatisfiable if and only if $\theta(x)$ is bisimulation invariant. Thus, $\mathrm{Pbl}_{\mathcal{FO}(\hat{\mathcal{L}})}(T_{Z}, \tilde{\theta})$ holds if and only if $\mathrm{Pbl}_{\mathcal{FO}}(\psi, \bot)$ holds.
\end{proof}

\begin{cor}\label{plrarrowpra}
Lemma~\ref{plrarrow} is also provable in $\mathrm{PRA}$.
\end{cor}

\begin{proof}
A $\Pi^0_2$ sentence expresses Lemma~\ref{plrarrow}. Thus, by the fact that $\mathrm{WKL}_0$ is a $\Pi^0_2$-conservative extension of $\mathrm{PRA}$, Lemma~\ref{plrarrow} can also be proved in  $\mathrm{PRA}$.
\end{proof}

\begin{prop}\label{vanBequivWC}
The following statements are equivalent over $\mathrm{RCA}_0$:
\begin{enumerate}
    \item The van Benthem characterization theorem (the hybrid form),
    \item Weak Completeness Theorem for $\mathcal{FO}$.
\end{enumerate}
\end{prop}

\begin{proof}
$(2)\Rightarrow(1)$: By Theorem~\ref{vanB}.
\vspace{1mm}

$(1)\Rightarrow(2)$: Let $\psi\in\mathrm{Snt}_{\mathcal{FO}}$ and let $T\subseteq\mathrm{Snt}_{\mathcal{FO}}$ be a finite set. Suppose that $T\cup\{\psi\}$ contains only $P_{0}(y), \dots, P_{k-1}(y)$, and $r(y, y^{\prime})$ as first-order predicates. We assume that $\neg\mathrm{Pbl}_{\mathcal{FO}}(T\cup\{\neg\psi\}, \bot)$ holds. Let $\phi\equiv\bigwedge T\land\neg\psi$ and $\theta(x)\equiv\phi^{P_{k}}\land r(x, x)\land\neg P_{k}(x)$. Then, by Corollary~\ref{plrarrowpra}, $\neg\mathrm{Pbl}_{\mathcal{FO}(\hat{\mathcal{L}})}(T_{Z}, \tilde{\theta})$ holds. Thus, by the van Benthem characterization theorem (the hybrid form), $\theta(x)$ is not bisimulation invariant. Thus, there exists a model $M=(W, R, V)$ on $\mathcal{L}_{\mathcal{FO}}$ and $w\in W$ such that $M\vDash\theta(w)$, that is, $M\vDash\phi^{P_{k}}$. We define $M^{\prime}=(W^{\prime}, R^{\prime}, V^{\prime})$ as follows:
\begin{itemize}
    \item $W^{\prime}=\{w\in W\mid P_{k}(w)\}$,
    \item $R^{\prime}=R\cap(W^{\prime}\times W^{\prime})$,
    \item $V^{\prime}(\varphi(\vec{a}))=V(\varphi^{P_{k}}(\vec{a}))$ for all $\varphi(\vec{x})\in\mathrm{Fml}_{\mathcal{FO}}$, and $\vec{a}\in W^{\prime}$.
\end{itemize}

Then $M^{\prime}\vDash\phi$. Thus,  $\bigwedge T\land\neg\psi$ has a model on $\mathcal{L}_{\mathcal{FO}}$.
\end{proof}

Next, we formulate Robinson arithmetic $\mathrm{Q}$ in the language $\mathcal{L}_{\mathcal{FO}}$. $\mathrm{Q}$ consists of a constant symbol $0$, an unary function symbol $S$, and binary function symbols $+,\,\cdot$. First, we formulate these symbols as relations in $\mathcal{L}_{\mathcal{FO}}$. Specifically, we adopt an event-based approach to these relations, where each arithmetic operation is associated with an explicit event variable and its arguments.
\vspace{2mm}

We introduce unary predicates $P_{\star}, P_{0}, P_{1}, \dots, P_{8}$. 
\vspace{2mm}

$(\mathrm{A})$: $P_{\star}(x)$ denotes that $x$ is a natural number.
\vspace{2mm}

$(\mathrm{B})$: The predicate $\mathrm{Succ}(x, y)$ denotes that $y=S(x)$: 
\begin{align*}
\mathrm{Succ}(x, y)&\equiv P_{\star}(x) \land P_{\star}(y) \land r(x, y).
\end{align*}

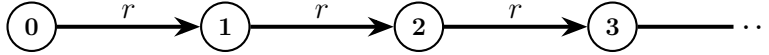
\begin{center}
\begin{tikzpicture}[
    scale=0.8, 
    transform shape,
    >=Stealth,
    thick,
    circle_node/.style={
        draw, 
        circle, 
        minimum size=8mm, 
        inner sep=0pt,
        fill=white, 
        font=\large\bfseries
    },
    box_node/.style={draw, rectangle, rounded corners=2pt, minimum height=6mm, inner sep=4pt},
    oval_node/.style={draw, ellipse, minimum width=2.5cm, minimum height=1cm}
]

    % --- 座標の定義 ---
    \coordinate (P0) at (0,1.5); 
    \coordinate (P1) at (3.2,1.5); 
    \coordinate (P2) at (6.4,1.5); 
    \coordinate (P3) at (9.6,1.5); 
    \coordinate (P4) at (11.6,1.5);

    % --- 1. 上部のタイムライン（数直線）と番号 ---
    \node[circle_node] (N0) at (P0) {0};
    \node[circle_node] (N1) at (P1) {1};
    \node[circle_node] (N2) at (P2) {2};
    \node[circle_node] (N3) at (P3) {3};

    \begin{scope}[on background layer]
        \draw[->, line width=1.5pt] (N0) -- node[above, midway, font=\Large] {$r$} (N1);
        \draw[->, line width=1.5pt] (N1) -- node[above, midway, font=\Large] {$r$} (N2);
        \draw[->, line width=1.5pt] (N2) -- node[above, midway, font=\Large] {$r$} (N3);
        \draw[line width=1.5pt] (N3) -- (P4) node[right, font=\large\bfseries, color=black] {\dots};
    \end{scope}

    % --- 注釈の追加 ---
    \node[below=0.2cm of N2, font=\small\itshape, text depth=0pt] {Note: Arrows between 0, 1, 2, 3, \dots\ represent successor relations.};

\end{tikzpicture}
\end{center}

$(\mathrm{C})$: The predicate $\mathrm{Add}(x, y, z, w, u_{1}, u_{2}, u_{3})$ denotes that $x+y=z$, where $w$ is the addition event, and $u_1, u_2, u_3$ represent that $x$, $y$, and $z$ are the first argument, the second argument, and the result of this event, respectively.
\begin{align*}
\mathrm{Add}(x, y, z, w, u_{1}, u_{2}, u_{3})\equiv P_{\star}(x)\land P_{\star}(y)\land P_{\star}(z)\land \Bigl(P_{1}(w)&\land(r(w, u_1) \land P_2(u_1) \land r(u_1, x))\\&\land(r(w, u_2) \land P_3(u_2) \land r(u_2, y))\\&\land(r(w, u_3) \land P_4(u_3) \land r(u_3, z))\Bigr).
\end{align*}

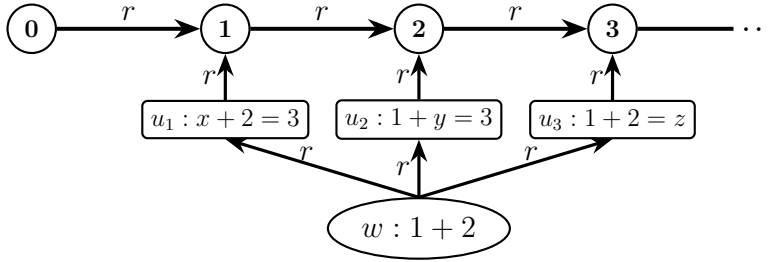
\begin{center}
\begin{tikzpicture}[
    scale=0.8, 
    transform shape,
    >=Stealth,
    thick,
    circle_node/.style={
        draw, 
        circle, 
        minimum size=8mm, 
        inner sep=0pt,
        fill=white, 
        font=\large\bfseries
    },
    box_node/.style={draw, rectangle, rounded corners=2pt, minimum height=6mm, inner sep=4pt},
    oval_node/.style={draw, ellipse, minimum width=2.5cm, minimum height=1cm}
]

    % --- 座標の定義 ---
    \coordinate (P0) at (0,1.5); 
    \coordinate (P1) at (3.2,1.5); 
    \coordinate (P2) at (6.4,1.5); 
    \coordinate (P3) at (9.6,1.5); 
    \coordinate (P4) at (11.6,1.5);

    % --- 1. 上部のタイムライン（数直線）と番号 ---
    \node[circle_node] (N0) at (P0) {0};
    \node[circle_node] (N1) at (P1) {1};
    \node[circle_node] (N2) at (P2) {2};
    \node[circle_node] (N3) at (P3) {3};

    \begin{scope}[on background layer]
        \draw[->, line width=1.5pt] (N0) -- node[above, midway, font=\Large] {$r$} (N1);
        \draw[->, line width=1.5pt] (N1) -- node[above, midway, font=\Large] {$r$} (N2);
        \draw[->, line width=1.5pt] (N2) -- node[above, midway, font=\Large] {$r$} (N3);
        \draw[line width=1.5pt] (N3) -- (P4) node[right, font=\large\bfseries, color=black] {\dots};
    \end{scope}

    % --- 2. 中段の引数・結果のボックス ---
    \node[box_node] (B1) at (3.2,0) {\large $u_1 : x+2=3$};
    \node[box_node] (B2) at (6.4,0) {\large $u_2 : 1+y=3$};
    \node[box_node] (B3) at (9.6,0) {\large $u_3 : 1+2=z$};

    % --- 3. 下部の数式オーバル ---
    \node[oval_node] (Eq) at (6.4,-1.8) {\Large $w : 1 + 2 $};

    % --- 4. 矢印の接続 ---
    \draw[->, line width=1.2pt] (B1.north) -- node[left, midway, font=\Large] {$r$} (N1.south);
    \draw[->, line width=1.2pt] (B2.north) -- node[left, midway, font=\Large] {$r$} (N2.south);
    \draw[->, line width=1.2pt] (B3.north) -- node[left, midway, font=\Large] {$r$} (N3.south);

    \draw[->, line width=1.2pt] (Eq.north) -- node[above left, midway, font=\Large] {$r$} (B1.south);
    \draw[->, line width=1.2pt] (Eq.north) -- node[left, midway, font=\Large] {$r$} (B2.south);
    \draw[->, line width=1.2pt] (Eq.north) -- node[above right, midway, font=\Large] {$r$} (B3.south);

    % --- 注釈の追加 ---
    \node[below=0.2cm of Eq, font=\small\itshape] {Note: This diagram illustrates the addition $1 + 2$.};

\end{tikzpicture}
\end{center}

$(\mathrm{D})$: The predicate $\mathrm{Mult}(x, y, z, w, u_{1}, u_{2}, u_{3})$ denotes that $x\cdot y=z$, where $w$ is the multiplication event, and $u_1, u_2, u_3$ represent that $x$, $y$, and $z$ are the first argument, the second argument, and the result of this event, respectively.
\begin{align*}
\mathrm{Mult}(x, y, z,w, u_{1}, u_{2}, u_{3})\equiv P_{\star}(x)\land P_{\star}(y)\land P_{\star}(z)\land\Bigl(P_{5}(w)&\land(r(w, u_1) \land P_6(u_1) \land r(u_1, x))\\&\land(r(w, u_2) \land P_7(u_2) \land r(u_2, y))\\&\land(r(w, u_3) \land P_8(u_3) \land r(u_3, z))\Bigr).
\end{align*}

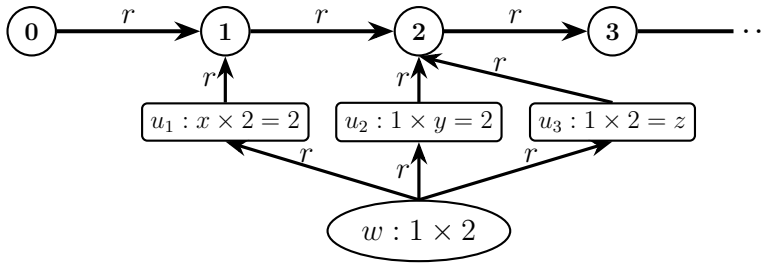
\begin{center}
\begin{tikzpicture}[
    scale=0.8, 
    transform shape,
    >=Stealth,
    thick,
    circle_node/.style={
        draw, 
        circle, 
        minimum size=8mm, 
        inner sep=0pt,
        fill=white, 
        font=\large\bfseries
    },
    box_node/.style={draw, rectangle, rounded corners=2pt, minimum height=6mm, inner sep=4pt},
    oval_node/.style={draw, ellipse, minimum width=2.5cm, minimum height=1cm}
]

    % --- 座標の定義 ---
    \coordinate (P0) at (0,1.5); 
    \coordinate (P1) at (3.2,1.5); 
    \coordinate (P2) at (6.4,1.5); 
    \coordinate (P3) at (9.6,1.5); 
    \coordinate (P4) at (11.6,1.5);

    % --- 1. 上部のタイムライン（数直線）と番号 ---
    \node[circle_node] (N0) at (P0) {0};
    \node[circle_node] (N1) at (P1) {1};
    \node[circle_node] (N2) at (P2) {2};
    \node[circle_node] (N3) at (P3) {3};

    \begin{scope}[on background layer]
        \draw[->, line width=1.5pt] (N0) -- node[above, midway, font=\Large] {$r$} (N1);
        \draw[->, line width=1.5pt] (N1) -- node[above, midway, font=\Large] {$r$} (N2);
        \draw[->, line width=1.5pt] (N2) -- node[above, midway, font=\Large] {$r$} (N3);
        \draw[line width=1.5pt] (N3) -- (P4) node[right, font=\large\bfseries, color=black] {\dots};
    \end{scope}

    % --- 2. 中段の引数・結果のボックス ---
    \node[box_node] (B1) at (3.2,0) {\large $u_1 : x\times2=2$};
    \node[box_node] (B2) at (6.4,0) {\large $u_2 : 1\times y=2$};
    \node[box_node] (B3) at (9.6,0) {\large $u_3 : 1\times2=z$};

    % --- 3. 下部の数式オーバル ---
    \node[oval_node] (Eq) at (6.4,-1.8) {\Large $w : 1 \times 2 $};

    % --- 4. 矢印の接続 ---
    \draw[->, line width=1.2pt] (B1.north) -- node[left, midway, font=\Large] {$r$} (N1.south);
    \draw[->, line width=1.2pt] (B2.north) -- node[left, midway, font=\Large] {$r$} (N2.south);
    \draw[->, line width=1.2pt] (B3.north) -- node[above left, midway, font=\Large] {$r$} (N2.south); 

    \draw[->, line width=1.2pt] (Eq.north) -- node[above left, midway, font=\Large] {$r$} (B1.south);
    \draw[->, line width=1.2pt] (Eq.north) -- node[left, midway, font=\Large] {$r$} (B2.south);
    \draw[->, line width=1.2pt] (Eq.north) -- node[above right, midway, font=\Large] {$r$} (B3.south);

    % --- 注釈の追加 ---
    \node[below=0.2cm of Eq, font=\small\itshape] {Note: This diagram illustrates the multiplication $1 \times 2$.};

\end{tikzpicture}
\end{center}

Then, we define the following axioms with the above predicates.
\vspace{2mm}

$(0)$: Each unary predicate $P_i$ is mutually exclusive with the others.
\begin{align*}
\forall x \ \neg(P_i(x) \land P_j(x)) \quad (i, j \in\{\star,1,\dots,8\}, \ i \neq j).
\end{align*}

$(1)$: There exists exactly one natural number zero.
\begin{align*}
\exists x (P_{\star}(x) \land P_0(x)),
\end{align*}
\begin{align*}
\forall x\forall y (P_{\star}(x) \land P_0(x)\land P_{\star}(y) \land P_0(y)\rightarrow x=y).
\end{align*}

$(2)$: $\mathrm{Succ}(x, y)$ defines a function on the natural numbers.
\begin{align*}
\forall x\exists y(P_{\star}(x)\rightarrow\mathrm{Succ}(x, y)),
\end{align*}
\begin{align*}
\forall x\forall y\forall y^{\prime}(\mathrm{Succ}(x, y)\land\mathrm{Succ}(x, y^{\prime})\rightarrow y=y^{\prime}).
\end{align*}

$(3)$: The addition event $w$ is connected to a node satisfying $P_i$.
\begin{align*}
\forall w \exists u_1 \Big( P_1(w) \to (P_i(u_1) \land r(w, u_1)) \Big)\,(i\in\{2, 3, 4\}).
\end{align*}

$(4)$: There is at most one node satisfying $P_i$ that is connected to the addition event $w$.
\begin{align*}
\forall w \forall u_1 \forall v_1 \Big( (P_1(w) \land P_i(u_1) \land r(w, u_1) \land P_i(v_1) \land r(w, v_1)) \to u_1 = v_1 \Big)\,(i\in\{2, 3, 4\}).
\end{align*}

$(5)$: Nodes satisfying $P_i$ are connected via $r$ to a node representing a natural number.
\begin{align*}
\forall w \forall u_1 \exists u_2 \Big( (P_1(w) \land P_i(u_1) \land r(w, u_1)) \to (P_{\star}(u_2) \land r(u_1, u_2)) \Big)\,(i\in\{2, 3, 4\}).
\end{align*}

$(6)$: If a node satisfying $P_i$ is connected via $r$ to a natural number, then such a natural number is unique.
\begin{align*}
\forall w \forall u_1 \forall u_2 \forall v_2 \Big( (P_1(w) \land P_i(u_1) \land r(w, u_1) \land P_{\star}(u_2) \land r(u_1, u_2) \land P_{\star}(v_2) \land r(u_1, v_2)) \to u_2 = v_2 \Big)\,(i\in\{2, 3, 4\}).
\end{align*}

$(7)$: For any natural numbers $x$ and $y$, there exist an addition event $w$ and nodes $u_1$, $v_1$ such that $u_1$ and $v_1$ take $x$ and $y$ as their first and second arguments, respectively.
\begin{align*}
\forall x \forall y \exists w \exists u_1 \exists v_1 \Bigg( (P_{\star}(x) \land P_{\star}(y)) \to \Big( P_1(w) \land (P_2(u_1) \land r(w, u_1) \land r(u_1, x)) \land (P_3(v_1) \land r(w, v_1) \land r(v_1, y)) \Big) \Bigg).
\end{align*}

$(8)$: The addition event $w$ for given inputs $x, y$ is unique.
\begin{align*}
\forall x \forall y \forall w \forall w' \forall u_1 \forall v_1 \forall u'_1 \forall v'_1 \Bigg(
& \Big[ P_1(w) \land (P_2(u_1) \land r(w, u_1) \land r(u_1, x)) \land (P_3(v_1) \land r(w, v_1) \land r(v_1, y)) \Big] \\
& \land \\
& \Big[ P_1(w') \land (P_2(u'_1) \land r(w', u'_1) \land r(u'_1, x)) \land (P_3(v'_1) \land r(w', v'_1) \land r(v'_1, y)) \Big] \\
& \to w = w'\Bigg).
\end{align*}

$(9)$: For any natural numbers $x$ and $y$, there exist a sum $z$ and nodes $w, u_1, u_2,$ and $u_3$ that connects them.
\begin{align*}
\forall x \forall y \exists z \exists w \exists u_1 \exists u_2 \exists u_3 \Big( (P_{\star}(x) \land P_{\star}(y)) \to \mathrm{Add}(x, y, z, w, u_1, u_2, u_3) \Big).
\end{align*}

$(10)$: If $z$ and $z'$ are two results of the computation of $x+y$, they must coincide.
\begin{align*}
\forall x \forall y \forall z \forall z' \forall w \forall w' \forall u_1 \forall u_1' \forall u_2 \forall u_2' \forall u_3 \forall u_3' \Bigg(
& \Big[ \mathrm{Add}(x, y, z, w, u_1, u_2, u_3) \land \mathrm{Add}(x, y, z', w', u_1', u_2', u_3' ) \Big] \\
& \to z = z'
\Bigg).
\end{align*}

$(11)$: The multiplication event $w$ is connected to a node satisfying $P_i$.
\begin{align*}
\forall w \exists u_1 \Big( P_5(w) \to (P_i(u_1) \land r(w, u_1)) \Big)\,(i\in\{6, 7, 8\}).
\end{align*}

$(12)$: There is at most one node satisfying $P_i$ that is connected to the multiplication event $w$.
\begin{align*}
\forall w \forall u_1 \forall v_1 \Big( (P_5(w) \land P_i(u_1) \land r(w, u_1) \land P_i(v_1) \land r(w, v_1)) \to u_1 = v_1 \Big)\,(i\in\{6, 7, 8\}).
\end{align*}

$(13)$: Nodes satisfying $P_i$ are connected via $r$ to a node representing a natural number.
\begin{align*}
\forall w \forall u_1 \exists u_2 \Big( (P_5(w) \land P_i(u_1) \land r(w, u_1)) \to (P_{\star}(u_2) \land r(u_1, u_2)) \Big)\,(i\in\{6, 7, 8\}).
\end{align*}

$(14)$: If a node satisfying $P_i$ is connected via $r$ to a natural number, then such a natural number is unique.
\begin{align*}
\forall w \forall u_1 \forall u_2 \forall v_2 \Big( (P_5(w) \land P_i(u_1) \land r(w, u_1) \land P_{\star}(u_2) \land r(u_1, u_2) \land P_{\star}(v_2) \land r(u_1, v_2)) \to u_2 = v_2 \Big)\,(i\in\{6, 7, 8\}).
\end{align*}

$(15)$: For any natural numbers $x$ and $y$, there exists at least one set consisting of an addition event $w$ and $P_6$-node $u_1$, and  $P_7$-node $v_1$ connected to $w$, such that $u_1$ and $v_1$ take $x$ and $y$ as arguments.
\begin{align*}
\forall x \forall y \exists w \exists u_1 \exists v_1 \Bigg( (P_{\star}(x) \land P_{\star}(y)) \to \Big( P_5(w) \land (P_6(u_1) \land r(w, u_1) \land r(u_1, x)) \land (P_7(v_1) \land r(w, v_1) \land r(v_1, y)) \Big) \Bigg).
\end{align*}

$(16)$: The multiplication event $w$ for given inputs $x, y$ is unique.
\begin{align*}
\forall x \forall y \forall w \forall w' \forall u_1 \forall v_1 \forall u'_1 \forall v'_1 \Bigg(
& \Big[ P_5(w) \land (P_6(u_1) \land r(w, u_1) \land r(u_1, x)) \land (P_7(v_1) \land r(w, v_1) \land r(v_1, y)) \Big] \\
& \land \\
& \Big[ P_5(w') \land (P_6(u'_1) \land r(w', u'_1) \land r(u'_1, x)) \land (P_7(v'_1) \land r(w', v'_1) \land r(v'_1, y)) \Big] \\
& \to w = w'\Bigg).
\end{align*}

$(17)$: For any natural numbers $x$ and $y$, there exist a product $z$ and nodes $w, u_1, u_2,$ and $u_3$ that connects them.
\begin{align*}
\forall x \forall y \exists z \exists w \exists u_1 \exists u_2 \exists u_3 \Big( (P_{\star}(x) \land P_{\star}(y)) \to \mathrm{Mult}(x, y, z, w, u_1, u_2, u_3) \Big).
\end{align*}

$(18)$: If $z$ and $z'$ are two results of the computation of $x\cdot y$, they must coincide.
\begin{align*}
\forall x \forall y \forall z \forall z' \forall w \forall w' \forall u_1 \forall u_1' \forall u_2 \forall u_2' \forall u_3 \forall u_3' \Bigg(
& \Big[ \mathrm{Mult}(x, y, z, w, u_1, u_2, u_3) \land \mathrm{Mult}(x, y, z', w', u_1', u_2', u_3' ) \Big] \\
& \to z = z'
\Bigg).
\end{align*}

$(19)$: The zero is not the successor of any number $\bigl(\forall x(S(x) \neq 0)\bigr)$.
\begin{align*}
\forall x \forall y \big( \mathrm{Succ}(x, y) \to \neg P_0(y) \big).
\end{align*}

$(20)$: The successor function is injective. $\bigl(\forall x \forall y(S(x) = S(y) \to x = y)\bigr)$.
\begin{align*}
\forall x \forall y \forall z \big( (\mathrm{Succ}(x, z) \land \mathrm{Succ}(y, z)) \to x = y \big).
\end{align*}

$(21)$: Every number other than zero is a successor. $\bigl(\forall x(x \neq 0 \to \exists y(x = S(y)))\bigr)$.
\begin{align*}
\forall x \exists y\big( (P_\star(x) \land \neg P_0(x)) \to (P_\star(y) \land \mathrm{Succ}(y, x)) \big).
\end{align*}

$(22)$: The base case of the addition $\bigl(\forall x(x + 0 = x)\bigr)$.
\begin{align*}
\forall x \forall y \exists w \exists u_1 \exists u_2 \exists u_3\big( (P_\star(x) \land P_\star(y) \land P_0(y)) \to \mathrm{Add}(x, y, x, w, u_1, u_2, u_3) \big).
\end{align*}

$(23)$: The recursive definition of the addition $\bigl(\forall x \forall y(x + S(y) = S(x + y))\bigr)$.
\begin{align*}
\forall x \forall y \forall y' \forall z \forall z' \forall w \forall w' \forall u_1 \forall u_2 \forall u_3 \forall v_1 \forall v_2 \forall v_3 
\Bigg(
& \Big[ \mathrm{Succ}(y, y') \land\mathrm{Add}(x, y, z, w, u_1, u_2, u_3) \\
& \quad \land \mathrm{Add}(x, y', z', w', v_1, v_2, v_3) \Big] \\
& \to \mathrm{Succ}(z, z')
\Bigg).
\end{align*}

$(24)$: The base case of the multiplication $\bigl(\forall x(x \cdot 0 = 0)\bigr)$.
\begin{align*}
\forall x \forall y \exists w \exists u_1 \exists u_2 \exists u_3\big( (P_\star(x) \land P_\star(y) \land P_0(y)) \to \mathrm{Mult}(x, y, y, w, u_1, u_2, u_3) \big).
\end{align*}

$(25)$: The recursive definition of multiplication $\bigl(\forall x \forall y(x \cdot S(y) = (x \cdot y) + x)\bigr)$.
\begin{align*}
\forall x \forall y \forall y' \forall z \forall w \forall m_1 \forall u_1 \forall u_2 \forall u_3 \forall m_2 \forall v_1 \forall v_2 \forall v_3 \exists a \exists s_1 \exists s_2 \exists s_3 
\Bigg(
& \Big[ \mathrm{Succ}(y, y')\land \mathrm{Mult}(x, y, z, m_1, u_1, u_2, u_3) \\
& \quad \land \mathrm{Mult}(x, y', w, m_2, v_1, v_2, v_3) \Big] \\
& \to \mathrm{Add}(z, x, w, a, s_1, s_2, s_3)
\Bigg).
\end{align*}

We define $\mathrm{Q}_{K}$ as the set of formulas consisting of the axioms $(0)$ through $(25)$. $\mathrm{Q}_{K}$ is an interpretation of $\mathrm{Q}_{rel}$, which is a relational formulation of Robinson arithmetic $\mathrm{Q}$. 

\begin{lem}[$\mathrm{RCA}_0$]\label{QKcon}
$\mathrm{Q}_{K}$ is consistent, that is, $\neg\mathrm{Pbl}_{\mathcal{FO}}\bigl(\mathrm{Q}_{K}, \bot\bigr)$.
\end{lem}
\begin{proof}
By the strong soundness theorem of first-order logic (see Simpson~\cite{MR2517689}), it is enough to show that $\mathrm{Q}_{K}$ has a weak model on $\mathcal{L}_{\mathcal{FO}}$. Let the domain of this model be $\mathbb{N}$, and consider an evaluation function restricted to the subformulas of $\mathrm{Q}_{K}$. Here, we assign the truth value to all formulas obtained by translating $\mathrm{Q}_{rel}$ into the language $\mathcal{L}_{\mathcal{FO}}$ using the predicates employed to construct $\mathrm{Q}_{K}$. This evaluation function can be constructed by using $\Sigma^0_1$-induction.
\end{proof}

We consider the $\Sigma_1$-completeness of $\mathrm{Q}_{K}$. For an $\mathcal{L}_{1}$-formula $\varphi$, let $\varphi^{\pi_{Q_{K}}}$ denote its interpretation via the predicates used in the construction of $\mathrm{Q}_{K}$.

\begin{lem}[The $\Sigma_1$-completeness of $\mathrm{Q}_{K}$, $\mathrm{RCA}_0$]\label{SgcomQK}
Let $\Sigma_{1}$-sentence. $\varphi$. If $\varphi$ holds, then $\mathrm{Pbl}_{\mathcal{FO}}\bigl(\mathrm{Q}_{K}, \varphi^{\pi_{Q_{K}}}\bigr)$.
\end{lem}

\begin{proof}
This follows from the fact that $\mathrm{Q}_{K}$ is an interpretation of $\mathrm{Q}_{rel}$ and that the $\Sigma_1$-completeness of $\mathrm{Q}$ is provable in $\mathrm{I}\Sigma^0_1$ (see Hájek and Pudlák~\cite{MR1219738}).
\end{proof}

\begin{state}[Weak $\Sigma^0_1$-separation~\cite{MR2517689}]
Let $\varphi(x)$ and $\psi(x)$ be $\Sigma^{0}_{1}$-formulas without set parameters, in which $X$ does not occur as a free variable. If $\neg\exists n(\varphi(n)\land\psi(n))$, then
\begin{align*}
\exists X\forall n((\varphi(n)\rightarrow n\in X)\land(\psi(n)\rightarrow n\not\in X)).
\end{align*}
\end{state}

Last, we show that the hybrid form is equivalent to the weak $\Sigma^0_1$-separation over $\mathrm{RCA}_0$. In fact, over $\mathrm{RCA}_0$, these claims are strictly stronger than $\mathrm{RCA}_0$ but strictly weaker than $\mathrm{WKL}_0$.

\begin{thm}
The following are equivalent over $\mathrm{RCA}_0$:
\begin{enumerate}
    \item The van Benthem characterization theorem (the hybrid form),
    \item Weak Completeness Theorem for $\mathcal{FO}$,
    \item Weak $\Sigma^0_1$-separation.
\end{enumerate}
\end{thm}

\begin{proof}
By Proposition~\ref{vanBequivWC}, $(1)\Leftrightarrow (2)$. Given the connection between $\mathrm{WKL}_0$ and G\"odel's completeness theorem, $(3)$ implies $(2)$ (See Simpson~\cite{MR2517689}). We show that $(2)$ implies $(3)$. By Lemma~\ref{QKcon} and Weak Completeness Theorem for $\mathcal{FO}$, $\mathrm{Q}_{K}$ has a model $M$ on $\mathcal{L}_{\mathcal{FO}}$. Let $T_{\mathrm{Q}_{K}}=\{\varphi\in\mathrm{Snt}_{\mathcal{FO}}\mid M\vDash\varphi\}$. Then $T_{\mathrm{Q}_{K}}$ is a consistent completion of $\mathrm{Q}_{K}$. By Lemma~\ref{SgcomQK}, the desired set $X$ can be constructed using $T_{\mathrm{Q}_{K}}$.
\end{proof}

\section*{Acknowledgements}
The work of the second author is partially supported by
 JSPS KAKENHI grant numbers JP21KK0045, JP23K03193, and 26K00615.

%\section{Some properties for modal logic in second-order arithmetic}

%\section{Open questions}

\bibliography{refmo}
\begin{comment}
\newpage

\renewcommand{\thesection}{\Alph{section}}
\setcounter{section}{1}

\section*{Appendix A: Some proofs}

\subsection{Proofs of Lemmas in Section~3}

%\subsection{Proofs of Lemmas in Section~4}

%\subsection{Proof of Lemma~\ref{frightk}}

\stepcounter{section}
\end{comment}
\end{document}